%% file: emg.tex
\RequirePackage[l2tabu,orthodox]{nag} %
\documentclass[english,abstract=true,captions=tableheading]{scrartcl}
\input{packages}
\KOMAoptions{DIV=last}
\input{commands}
\hypersetup{
pdftitle={Bayesian inversion for electromyography using low-rank tensor formats},
pdfauthor={Anna R\"orich, Tim A. Werthmann, Dominik G\"oddeke, Lars Grasedyck}, 
pdfkeywords={inverse problems, parameter-dependent problem, Metropolis-Hastings algorithm, hierarchical Tucker format}
}
\begin{document}	
\title{%
Bayesian inversion for electromyography using low-rank tensor formats
}
\author{%
Anna R{\"o}rich\thanks{%
Institute~of~Applied~Analysis~and~Numerical~Simulation, University~of~Stuttgart, Allmandring~5b, 70569~Stuttgart, Germany\newline%
\href{mailto:anna.roerich@ians.uni-stuttgart.de}{anna.roerich@ians.uni-stuttgart.de}, \href{mailto: dominik.goeddeke@ians.uni-stuttgart.de}{dominik.goeddeke@ians.uni-stuttgart.de}%
}%
\and%
Tim A.\@ Werthmann\thanks{%
Institut~f{\"u}r~Geometrie~und~Praktische~Mathematik, RWTH~Aachen~University, Templergraben~55, 52056~Aachen, Germany\newline%
\href{mailto:werthmann@igpm.rwth-aachen.de}{werthmann@igpm.rwth-aachen.de}, \href{mailto:lgr@igpm.rwth-aachen.de}{lgr@igpm.rwth-aachen.de}%
}%
\and%
Dominik G{\"o}ddeke\footnotemark[1]
\thanks{Stuttgart~Center~for~Simulation~Science, University~of~Stuttgart, Pfaffenwaldring~5a, 70569~Stuttgart, Germany}
\and%
Lars Grasedyck\footnotemark[2]%
}%
\date{\today}
\maketitle
\begin{abstract}
The reconstruction of the structure of biological tissue using electromyographic data is a non-invasive imaging method with diverse medical applications.
Mathematically, this process is an inverse problem. 
Furthermore, electromyographic data are highly sensitive to changes in the electrical conductivity that describes the structure of the tissue.
Modeling the inevitable measurement error as a stochastic quantity leads to a Bayesian approach. 
Solving the discretized Bayesian inverse problem means drawing samples from the posterior distribution of parameters, e.g., the conductivity, given measurement data.
Using, e.g., a Metropolis-Hastings algorithm for this purpose involves solving the forward problem for different parameter combinations which requires a high computational effort.
Low-rank tensor formats can reduce this effort by providing a data-sparse representation of all occurring linear systems of equations simultaneously and allow for their efficient solution.
The application of Bayes' theorem proves the well-posedness of the Bayesian inverse problem.
The derivation and proof of a low-rank representation of the forward problem allow for the precomputation of all solutions of this problem under certain assumptions, resulting in an efficient and theory-based sampling algorithm.
Numerical experiments support the theoretical results, but also indicate that a high number of samples is needed to obtain reliable estimates for the parameters.
The Metropolis-Hastings sampling algorithm, using the precomputed forward solution in a tensor format, draws this high number of samples and therefore enables solving problems which are infeasible using classical methods.
\par\vspace\baselineskip\noindent
\end{abstract}
\input{tex/introduction}
\input{tex/electromyographic_model}

\input{tex/problem}
\input{tex/representation_parameter_dependent_problems}
\input{tex/cp_format}

\input{tex/ht_format}
\input{tex/Bayes_inverse_EMG}
\input{tex/linear_system}
\input{tex/finite_difference_operator}
\input{tex/finite_difference_rhs}
\input{tex/parameter_dependent_problems}
\input{tex/operator}
\input{tex/algorithms}

\input{tex/numerical_experiments}
\input{tex/related_work}
\input{tex/conclusion}
\section*{Acknowledgments}
We thank Maren Klever for her critical reading of and suggestions for this article.

We thank the anonymous referees for helping to improve the article by their suggestions.

This research was partially funded by the Deutsche Forschungsgemeinschaft (DFG, German Research Foundation) under Germany's Excellence Strategy -- EXC-2075 -- 390740016.

L.\@ Grasedyck and T. \@A.\@ Werthmann have been supported by the DFG within the DFG priority program 1886 (SPPPoly) under Grant No.\@ GR-3179/5-1.
\printbibliography{}
\end{document}

%% file: packages.tex
\usepackage[utf8]{inputenc} %
\usepackage[T1]{fontenc}  %
\usepackage{lmodern} %
\usepackage[final]{microtype}
\usepackage[english]{babel}
\usepackage{csquotes}
\usepackage{amsmath,amssymb,amsthm}  %

\usepackage[backend=biber,style=numeric,maxbibnames=5,giveninits=true,hyperref=true,isbn=false,url=false,doi=true,date=year]{biblatex}
\usepackage{booktabs} %
\usepackage[binary-units=true]{siunitx}
\usepackage{algorithm}
\usepackage{algorithmic}

\usepackage{graphicx}
\usepackage{subcaption}
\usepackage{hyperref}  %
\usepackage{bookmark}
\usepackage[noabbrev]{cleveref}
\usepackage{enumitem}
\usepackage{tikz}
\usepackage{pgfplots}
\usepackage{pgfplotstable}
\pgfplotsset{compat=newest} 
\pgfplotsset{plot coordinates/math parser=false}
\usepackage{xcolor}
\usepackage{qtree}
\usepackage{hyphenat}
\usepackage[misc]{ifsym}

\usepackage{mathtools}
\usepackage[textsize=tiny]{todonotes}

%% file: commands.tex
\newcommand{\norm}[1]{\left\lVert#1\right\rVert}
\newcommand{\Id}{\operatorname{Id}}  %
\newcommand{\diag}{\operatorname{diag}}  %
\newcommand{\I}{\mathcal{I}}  %
\newcommand{\N}{\mathbb{N}} %
\newcommand{\R}{\mathbb{R}} %
\newcommand{\HT}{\mathcal{H}\operatorname{-Tucker}} 

\newcommand{\righthandside}{right\hyp{}hand~side}

\newcommand{\lowrank}{low\hyp{}rank}
\newcommand{\treesymbol}{z}
\newcommand{\AlgA}{\mathbf{A}} 
\newcommand{\Algb}{\mathbf{b}}
\newcommand{\AlgM}{\mathbf{M}}
\newcommand{\Algphi}{\boldsymbol{\phi}}   
\newcommand{\AlgR}{\boldsymbol{\rho}}
\newcommand{\AlgZ}{\boldsymbol{\zeta}}
\newcommand{\AlgP}{\boldsymbol{\pi}}
\newcommand{\AlgQ}{\boldsymbol{\theta}}
\theoremstyle{plain}
\newtheorem{theorem}{Theorem}[section] 
\newtheorem*{theorem*}{Theorem}
\newtheorem{lemma}[theorem]{Lemma} 
\newtheorem*{lemma*}{Lemma} 
\newtheorem{corollary}[theorem]{Corollary} 
\newtheorem*{corollary*}{Corollary} 
 
\newtheorem*{proposition*}{Proposition}
 
\newtheorem*{conjecture*}{Conjecture}
 
\newtheorem*{criterion*}{Criterion}
 
\newtheorem*{assertion*}{Assertion}
\theoremstyle{definition}
\newtheorem{definition}[theorem]{Definition}
\newtheorem*{definition*}{Definition} 

\newtheorem*{condition*}{Condition} 

\newtheorem*{problem*}{Problem} 

\newtheorem*{example*}{Example} 

\newtheorem*{exercise*}{Exercise} 

\newtheorem*{question*}{Question} 

\newtheorem*{axiom*}{Axiom} 

\newtheorem*{property*}{Property} 

\newtheorem*{assumption*}{Assumption} 

\newtheorem*{hypothesis*}{Hypothesis} 
\theoremstyle{remark} 
\newtheorem{remark}[theorem]{Remark}
\newtheorem*{remark*}{Remark}

\newtheorem*{note*}{Note}

\newtheorem*{notation*}{Notation}

\newtheorem*{claim*}{Claim}

\newtheorem*{summary*}{Summary}

\newtheorem*{acknowledgment*}{Acknowledgment}

\newtheorem*{case*}{Case}

\newtheorem*{conclusion*}{Conclusion}

%% file: tex/introduction.tex
\section{Introduction}\label{sec:Introduction}
In clinical applications, surface electromyographic~(EMG) data are a widely used source of information about the muscular and nervous system. 
For example, EMG data are a valuable source of information in neurology, movement analysis, rehabilitation medicine or the development of biofeedback techniques.
To this end, different models have been developed to simulate and understand EMG data, see, e.g.,~\cite{R15}.

Using EMG measurements, we focus on reconstructing the intracellular conductivity of biological tissue. 
As the conductivity provides information about the structure of this tissue, we make an important step towards a non-invasive and radiation-free imaging method. 
Furthermore, reliable estimates on the conductivity from patient-specific EMG measurements can advance the personalized treatment. 

Computed EMG data is, however, highly sensitive to changes in the conductivity, see, e.g.,~\cite{R16}. 
In addition, reconstructing data from (surface) measurements is an inverse problem~\cite{R10}. 
Since the measurement error is unknown, we model it as a stochastic quantity and include it into the EMG model.
This results in a probabilization of the whole EMG model. 
Consequently, the solution of the inverse EMG problem also becomes probabilistic. 

For solving this probabilistic inverse problem, in Section~\ref{sec:EMG_model}, we use a Bayesian ansatz, cf.~\cite{R2,R11}, that searches for the probability distribution of the parameters for given measurements, the so-called \emph{posterior distribution}. 
This ansatz has the advantage that the posterior distribution quantifies the uncertainty within instances of the reconstructed parameters. 

Discretizing the posterior distribution means drawing a finite number of samples from the posterior which includes solving the (discrete) forward EMG problem for different parameter samples to check the fidelity of each sample. 

As solving the forward EMG problem is expensive using classical methods, we aim at precomputing the solution of the forward problem for all parameters at the same time.
This results in a parameter-dependent linear system of equations, i.e., \( A(p) \phi(p) = b(p) \) for an operator \(A\), a solution \(\phi \), and a right-hand side \(b\) depending on parameters \(p = (p^{(1)}, p^{(2)}, \dots, p^{(d)})\). 
After discretizing the parameters in the sense that we allow each parameter \(p^{(j)}\), \(j = 1,\ldots,d\), to take \(n\) different values from its domain, solving the linear system for every combination of parameters implies solving \(n^d\) linear systems. 
This exponential scaling in the dimension \(d \) of the parameter space is commonly known as the \emph{curse of dimensionality} which renders classical methods for \(d \gg 2\) infeasible.

To represent these parameter-dependent linear systems, we use low-rank tensor formats, cf.~\cite{Grasedyck2013,Hackbusch2012}, which we recapitulate in Section~\ref{sec:repParameterProblems}.
Solving these linear systems within these formats allows us to evaluate the parameter-dependent forward problem fast. 

In particular, our main contributions to solve this Bayesian inverse EMG problem and to represent the forward problem in a data-sparse way using low-rank tensor formats are:
\begin{itemize}
\item We prove the well-posedness of our particular Bayesian inverse EMG problem in Section~\ref{sec:Bayes_inverse} and show that modeling the measurement error leads to a natural regularization of the inverse problem.
\item We derive a discretization of the parameter-dependent operator and the right-hand side in Section~\ref{sec:guiding_example} and prove a data-sparse representation of this discretization using low-rank tensor formats.
This method allows us to solve the parameter-dependent linear system fast.
\item Combining this data-sparse representation with a standard Metropolis-Hastings algorithm in Section~\ref{sec:algos} allows us to solve the Bayesian inverse EMG problem efficiently.
\end{itemize}
In Section~\ref{sec:NumExp}, we present our numerical experiments that support our theoretical analysis and indicate that the Markov chain constructed by the Metropolis-Hastings algorithm using low-rank tensor formats behaves like the Markov chain constructed by a standard algorithm.
Further, we observe a speedup of more than \(600\) using low-rank tensor formats compared to a standard algorithm.

In Section~\ref{sec:relatedwork}, we discuss some related work, and in Section~\ref{sec:conclusion}, we conclude that mathematical theory and an efficient representation of the parameter-dependent solution, which allows us to generate samples fast, leads to an efficient algorithm to solve the Bayesian inverse problem.

%% file: tex/electromyographic_model.tex
\section{The Bayesian inverse electromyographic problem}\label{sec:EMG_model}
In order to define our Bayesian inverse EMG problem, we briefly discuss the structure of skeletal muscles and summarize a forward model of surface EMG signals in the following.

A skeletal muscle is composed of bundles of cells, the so-called muscle fibers. 
These muscle fibers are the active contractile tissue of a body that react to electrical stimuli.
To model
EMG signals, we follow the physical structure of a skeletal muscle beginning with the electrical behavior of a single muscle fiber and then describing the electrical behavior of a skeletal muscle by assembling the muscle fibers.

An electrical stimulus from the spinal cord influences the chemo-electrical behavior of the innervated muscle fibers \(D_{\text{F},j} \subseteq \mathbb{R}\), \(j =1,\ldots, N_{\text{MF}}\), for \(N_{\text{MF}}\in\mathbb{N}\) muscle fibers.
These electrical fluctuations travel along the muscle fibers as action potentials~(APs), propagate through the muscle,
and are measured at \(M\in\mathbb{N}\) measuring points summarized in \(\mathbf{x} \in \mathbb{R}^{M \times 3}\). 

We apply the widely used model by Rosenfalck~\cite{R8} to model the muscle fiber AP\@:
\begin{equation}\label{eq:Rosenfalck}
	v_{\text{m},j}(s) = r_{1,j} s^3 \exp(-r_{2,j} s) - r_{3,j} \quad \text{ for }s \in D_{\text{F},j},\ j = 1,\ldots,N_{\text{MF}}. 
\end{equation}
Here, \(r_{1,j},r_{2,j},r_{3,j} \in \mathbb{R}\) are known, fixed constants, and the spatial coordinate \(s\) can be rewritten as \(s = u_j t\) using the AP velocities \(u_j\) and time \(t\). 

To assemble a three-dimensional skeletal muscle \(D_\text{M} \subseteq \mathbb{R}^3\) from the one-dimensional muscle fibers \(D_{\text{F},j} \subseteq \mathbb{R}\), a transfer operator is needed.
Thus, we introduce the smoothing operator \(S: D_{\text{F},j} \to \mathbb{R}^3\) with
\begin{equation}\label{eq:smoothing_operator}
	S(v_{\text{m},j})(x) = v_{\text{m},j}(\pi_j(x)) \exp\left(-\frac\beta2\norm{x - \pi_j(x)}_{\mathbb{R}^3}^2\right) ,
\end{equation}
where \( \beta \in \mathbb{R} \) is a smoothing parameter and \(\pi_j: D_\text{M} \to D_{\text{F},j}\) is the orthogonal projection of a muscle tissue point \(x \in D_\text{M}\) onto the muscle fiber \(D_{\text{F},j}\) with starting point \(y_j \in \mathbb{R}^3\) and direction \(\vec{d}_j \in \mathbb{R}^3\). 
Note that the muscle fiber directions \(\vec{d}_j\) in general depend on \( x \in \mathbb{R}^3 \) and are known for the forward problem, e.g., through a medical imaging technique.
The projection reads
\begin{equation}\label{eq:projection}
	\pi_j(x) = y_j + \frac{{(x - y_j)}^\top \vec{d}_j}{\vec{d}_j^\top \vec{d}_j}\vec{d}_j .
\end{equation}
Applying the smoothing operator to the muscle fibers yields \( \cup_{j=1}^{N_{\text{MF}}} S(D_{\text{F},j}) = D_\text{M} \), and we obtain the membrane potential \( V_\text{m}(x) = \sum_{j=1}^{N_{\text{MF}}} S(v_{\text{m},j})(x) \).

The bidomain equation, as stated in~\cite{R15}, models the propagation of the membrane potential \(V_\text{m}\) through a skeletal muscle by
\begin{equation}\label{eq:main_eq}
     \nabla\cdot( (\sigma_\text{i} + \sigma_\text{e})\nabla\phi_\text{e} ) = -\nabla\cdot(\sigma_\text{i}\nabla V_\text{m}) \quad \text{ in }  D_\text{M} , 
\end{equation}
where \(\phi_\text{e}\) denotes the extracellular electrical potential, and \(\sigma_\text{i}\), \(\sigma_\text{e}\) are the intra- and extracellular electrical conductivities.
Additionally, no-flow boundary conditions are introduced at the domain boundary.
A zero-mean integral condition is used to ensure uniqueness of the solution.

The above model can easily be extended by the electrophysiology of surrounding connective tissue and bones, and a model of force generation and the corresponding continuum mechanics, see~\cite{R15} and the references therein.
Within our setting, the muscle geometry and the structure of the tissue remain unchanged in time.

Note that we model the conductivities as matrices, e.g., \(\sigma_\text{i} \in \mathbb{R}^{3 \times 3}\),
where each matrix entry \({(\sigma_\text{i})}_{j,k}\) quantifies the conductivity of the tissue in the \(x_j\)-\(x_k\)-direction for \(j,k = 1,2,3\).
In particular, the eigenvector of \(\sigma_\text{i}\) that belongs to the largest eigenvalue represents the orientation of the underlying muscle fiber, and the largest eigenvalue corresponds to the longitudinal conductivity of the underlying muscle fiber.
Note that the conductivity of a muscle fiber in transversal direction is much smaller. 
This relation enables us to draw conclusions about the structure of muscular tissue from its intracellular conductivity.
To verify our ansatz described in the following sections, we restrict ourselves to diagonal conductivity matrices, i.e., the corresponding eigenvectors are the unit vectors \(\vec{e}_j \in \mathbb{R}^3\) for \(j=1,2,3\).
Consequently, the muscle fiber direction is one of these unit vectors.

A reasonable assumption on \(\sigma_\text{i}\) is that it is bounded, i.e., there exist constants \(s_{-} > 0\) and \(s_{+} < \infty \) such that \( s_{-} \le \sigma_\text{i} \le s_{+} \) holds componentwise.
Physically this corresponds to the tissue neither being fully insulating nor superconducting. 
Formalizing these considerations leads to the assumption \(p \coloneqq \left( {(\sigma_\text{i})}_{1,1}, {(\sigma_\text{i})}_{2,2}, {(\sigma_\text{i})}_{3,3} \right) \in {[s_{-},s_{+}]}^3 \eqqcolon \mathcal{J}\).

For simplicity, we encapsulate the above models in the definition of the observation operator
\begin{equation}\label{eq:observation_operator}
      \mathcal{G}_{\mathbf{x}} : \mathcal{J} \to \mathbb{R}^{M} \quad\text{ with } p \mapsto \phi(\mathbf{x}) ,
\end{equation}
which maps the diagonal entries \(p\) of a given intracellular conductivity \(\sigma_\text{i}\) to the calculated electrical potential \( \phi(\mathbf{x})\) at measuring points \(\mathbf{x} \in \mathbb{R}^{M \times 3}\). 

To complete the forward EMG model, we include the inevitable measurement error which is unknown but is usually assumed to be additive and to follow a normal distribution. 
Hence, the measurement error is modeled as a random variable \(\eta: \Omega \to \mathbb{R}^{M}\) on a complete probability space \((\Omega,\mathcal{F},P)\) with \(\eta\sim \mathcal{N}(0,\Xi)\) and covariance matrix \(\Xi = \operatorname{diag}(\xi,\ldots,\xi)\in\mathbb{R}^{M\times M}\). 
Adding the measurement error to~\eqref{eq:observation_operator} yields the model for EMG data
\begin{equation}\label{eq:measurement_error}
    \phi_\text{EMG}^\text{comp}(p) = \phi_\text{EMG}^\text{comp}(p,\mathbf{x},\omega) \coloneqq \mathcal{G}_\mathbf{x}(p) + \eta(\omega) \in\mathbb{R}^{M} .
\end{equation}

Solving~\eqref{eq:measurement_error} for \(p\), as in the inverse problem setting, shows that \(p\) must be a random variable as well. 
For emphasizing the randomness of \(p\), we write \(p = p(\omega)\).

A naive inversion of the probabilistic forward problem would be to search for a \(p(\omega) \in \mathcal{J}\) such that \(\phi_\text{EMG}^\text{comp}(p(\omega)) = \phi_\text{EMG}^\text{meas}\) for given measurements \(\phi_\text{EMG}^\text{meas}\in\mathbb{R}^{M}\). 
This problem formulation searches for particular realizations of the random variable \(p\) that, however, misrepresents the behavior of the probabilistic inverse EMG problem.
Hence, we need a more appropriate problem formulation.

We consider a function space Bayesian formulation which aims at calculating the probability distribution of \(p\) for given data \(\phi_\text{EMG}^\text{meas}\). 

To follow this approach, we assume that the entries of \(p\) are uncorrelated and equip \(\mathcal{J}\) with the product \(\sigma \)-algebra \( \Theta \coloneqq \bigotimes_{j = 1}^3\mathcal{B}({[s_{-},s_{+}]})\), where \(\mathcal{B}({[s_{-},s_{+}]})\) is the Borel-\(\sigma \)-algebra on \([s_{-},s_{+}]\). 
Subsequently, the product probability measure \( \rho \coloneqq \bigotimes_{j=1}^3 \mathrm{d}\lambda_j \) is defined on the measurable space \((\mathcal{J},\Theta)\) with \(\mathrm{d}\lambda_j\) denoting the normalized Lebesgue measure on \({[s_{-},s_{+}]}\), similar to~\cite{R3,R5}. 
Note that \(\rho \) is the probability law of the random variable \(p\), since the diagonal entries \(p\) of the intracellular conductivity \( \sigma_{\text{i}} \) are uncorrelated.
The Lebesgue measure indicates that the entries of \(p\) are uniformly distributed on \([s_{-},s_{+}]\).
In the Bayesian context, \(\rho \) is called the \emph{prior measure} or short \emph{prior}, because it describes the behavior of \(p\) prior to having any knowledge about the conductivity, e.g., from measurements.

%% file: tex/problem.tex
The Bayesian inverse EMG problem searches for the conditioned probability distribution \(\rho^\text{EMG}\) of \(p\) given EMG measurements \(\phi_\text{EMG}^\text{meas}\).
We prove the existence of the posterior distribution \(\rho^\text{EMG}\) in Section~\ref{sec:Bayes_inverse}.

For solving our Bayesian inverse EMG problem, we use a Metropolis-Hastings algorithm, see, e.g.,~\cite{norris_1997}.
A Metropolis-Hastings algorithm is an acceptance-rejection algorithm that draws samples from the posterior distribution by solving the EMG forward problem for different realizations of \(p\) and comparing the results. 
If the proposal is accepted by an acceptance strategy \(a\), it becomes part of a Markov chain. 
Otherwise, the old sample will be kept and a new proposal will be drawn.

In~\cite{R2}, the acceptance strategy \( a(p, \tilde{p}) \coloneqq \min\lbrace 1, \exp(\Phi(p) - \Phi(\tilde{p}))\rbrace \) with the potential \(\Phi:\mathcal{J} \times \mathbb{R}^{M} \to \mathbb{R}^{}\) defined by
\begin{equation}\label{eq:def_potential}
\Phi(p,\phi_\text{EMG}^\text{meas}) \coloneqq \frac12 \norm{ \phi_\text{EMG}^\text{meas} - \mathcal{G}_{\mathbf{x}}(p) }_\Xi^2 - \frac12 \norm{\phi_\text{EMG}^\text{meas}}_\Xi^2
\end{equation}
and \(\Xi \)-norm \(\norm{ \mathbf{v} }_\Xi \coloneqq \norm{ \Xi^{-\frac{1}{2}}\mathbf{v} }_{\mathbb{R}^{M}}\) for all \(\mathbf{v} \in \mathbb{R}^M\) was derived such that the resulting Markov chain is reversible with respect to the prior \(\rho \). This yields the convergence of the Metropolis-Hastings algorithm.

We rewrite the acceptance strategy:
\begin{align*}
     a(p, \tilde{p}) &= \min\big\lbrace1, \exp\big(\Phi(p) - \Phi(\tilde{p})\big)\big\rbrace \\
     &= \min\left\lbrace1, \frac{\exp\big(\frac12 \norm{ \phi_\text{EMG}^\text{meas} - \mathcal{G}_{\mathbf{x}}(p) }_\Xi^2 \big)}{\exp\big(\frac12 \norm{ \phi_\text{EMG}^\text{meas} - \mathcal{G}_{\mathbf{x}}(\tilde{p}) }_\Xi^2 \big)}\right\rbrace \\
     &\quad \begin{cases}
      = 1 \ \text{if }\ \norm{ \phi_\text{EMG}^\text{meas} - \mathcal{G}_{\mathbf{x}}(\tilde{p}) }_\Xi^2 \le \norm{ \phi_\text{EMG}^\text{meas} - \mathcal{G}_{\mathbf{x}}(p) }_\Xi^2 ,\\
      < 1 \ \text{otherwise} .
     \end{cases}
\end{align*}
Consequently, a new proposal will always be accepted, if it produces a smaller error than the last accepted sample, and will otherwise be rejected with probability \(1-a\), i.e., the old sample will be kept with probability \(1-a\). 

%% file: tex/representation_parameter_dependent_problems.tex
\section{Low-rank tensor formats}\label{sec:repParameterProblems}
Evaluating the acceptance strategy in every step of the Metropolis-Hastings algorithm requires the evaluation of the observation operator \(\mathcal{G}_{\mathbf{x}}\), i.e., the solution of the forward EMG problem, for a new set of diagonal entries \(p\) of the intracellular conductivity. 
Consequently, we need a way to compute these solutions fast.
We use low-rank tensor formats to accelerate these computations and motivate these formats using an example, analog to~\cite{Grasedyck2020}.

We consider the scaling of a discrete operator \(A_\text{h}\) by a parameter \(p_\text{h}(j)\), \( j = 1,\dots,n\) with \( n \in \mathbb{N} \), i.e., \(p_\text{h}(j) A_\text{h}\).
We assume that the \righthandside{} \( b_\text{h} \) is constant for all \(p_\text{h}(j)\). 
Using classical methods, we would need to solve the following linear system:
\begin{equation*}
	\begin{pmatrix}
			p_\text{h}(1) A_\text{h} & 0 & \hdots & 0 \\
			0 & p_\text{h}(2) A_\text{h} & \ddots & \vdots \\
			\vdots & \ddots & \ddots & 0 \\
			0 & \hdots  & 0 & p_\text{h}(n) A_\text{h}
	\end{pmatrix}
	\begin{pmatrix}
			\phi_\text{h}(p_\text{h}(1)) \\
			\phi_\text{h}(p_\text{h}(2)) \\
			\vdots \\
			\phi_\text{h}( p_\text{h}(n))
	\end{pmatrix}
	=
	\begin{pmatrix}
			b_\text{h} \\
			b_\text{h} \\
			\vdots \\
			b_\text{h}
	\end{pmatrix}.
\end{equation*} 
Using the Kronecker product to reformulate this system
\begin{equation*}
	\left( \diag(p_\text{h}(1), p_\text{h}(2), \dots, p_\text{h}(n)) \otimes A_\text{h} \right) \phi_\text{h}(p_\text{h}) = {\left(1, \dots, 1\right)}^\top \otimes b_\text{h} ,
\end{equation*}
we achieve a data-sparse representation.
We use a generalization of this representation to derive a data-sparse representation of the parameter-dependent forward EMG problem which can be interpreted as the \emph{CANDECOMP/PARAFAC}, or short CP, \emph{representation} introduced in~\cite{Carroll1970,Harshman1970}.

%% file: tex/cp_format.tex
\begin{definition}[CP vector and CP operator]\label{def:CPDecomposition}
A \emph{CP representation} of a tensor \(\mathbf{b}\in \R^{n_1 \times \cdots \times n_d}\), with \emph{representation rank} \(r \in \N_0\), is defined as
\begin{equation}\label{eq:CPDecomposition}
\mathbf{b} = \sum_{k = 1}^r \bigotimes_{\ell = 1}^d b_{k}^{\left( \ell \right)} \quad \text{ with } b_{k}^{\left( \ell \right)} \in \R^{n_{\ell}} .
\end{equation}
We call each \(\ell \in \mathcal{D} \coloneqq \lbrace 1, \dots, d \rbrace \) \emph{mode} and \(d\) the \emph{dimension}. 
The minimal \(r\), such that~\eqref{eq:CPDecomposition} holds, is called the \emph{CP rank} of \(\mathbf{b}\) and in this case~\eqref{eq:CPDecomposition} is called the \emph{CP decomposition} of \(\mathbf{b}\).
We call a tensor of the form~\eqref{eq:CPDecomposition} a \emph{CP vector}.

A CP representation of a tensor operator \(\mathbf{A} \) from \( \R^{n_1 \times \cdots \times n_d}\) to \( \R^{n_1 \times \cdots \times n_d}\), with representation rank \(r \) and dimension \(d\), is defined as
\begin{equation}\label{eq:CPDecompositionOperator}
\mathbf{A} = \sum_{k = 1}^r \bigotimes_{\ell = 1}^d A_{k}^{\left( \ell \right)} \quad \text{ with } A_{k}^{\left( \ell \right)} \in \R^{n_{\ell} \times n_{\ell} } .
\end{equation}
We call a tensor of the form~\eqref{eq:CPDecompositionOperator} a \emph{CP operator}.
\end{definition}
Note that the \( b_k^{(\ell)} \) in~\eqref{eq:CPDecomposition} are vectors and that the \( A_k^{(\ell)} \) in~\eqref{eq:CPDecompositionOperator} are matrices.
Therefore, a CP vector \( \mathbf{b} \) is a sum of rank \(r\) Kronecker products of \(d\) vectors, and a CP operator \( \mathbf{A} \) is a summation over Kronecker products of matrices.
Thus, using Definition~\ref{def:CPDecomposition}, there exist CP vectors and CP operators of any dimension and rank.

A big advantage of the CP format is the data-sparsity in case of a small representation rank \(r\), since a tensor \(\mathbf{b} \in \R^{n_1 \times \cdots \times n_d}\) of the form~\eqref{eq:CPDecomposition} has storage cost in \(\mathcal{O} (r \sum_{\ell=1}^d n_{\ell}  ) \approx \mathcal{O}(r d n)\) compared to \(\mathcal{O}( \prod_{\ell=1}^d n_{\ell} ) \approx \mathcal{O}(n^d)\) with \( n = \max_{\ell \in \mathcal{D}} n_\ell \).

Therefore it is desirable to represent the operator and the right-hand side of the forward EMG problem data-sparse using low-rank tensor formats.
To compute the solution of the discrete forward EMG problem, we need to solve linear systems within low-rank tensor formats.
An algorithm that can calculate the inverse of an operator with rank \(r>1\) in a direct way is unknown.

Consider, e.g., a CP operator \( \mathbf{A} \) of dimension \(1 \) and rank \( 2 \), i.e., \( \mathbf{A} = A_1 + A_2 \), with \(A_1, A_2 \in \mathbb{R}^{n \times n} \).
Then, finding a direct inverse of \( \mathbf{A} \) in the CP format means finding matrices \(C_j\) and \(D_j\) such that \( \mathbf{A}^{-1} = {(A_1 + A_2)}^{-1} = \sum_{j=1}^J C_j^{-1} + D_j^{-1} \) should hold for some rank \(J \in \mathbb{N}\).
Since such a property is unknown even for matrix summations~\cite{Miller1981}, it is also unknown in the more general tensor case.

We therefore need iterative solvers and thus arithmetic operations within low-rank tensor formats. 
These arithmetic operations often lead to an increase of the representation rank.

Consider, e.g., a CP operator of dimension \( 2 \) and rank \(3 \), i.e., \( \mathbf{A} = \sum_{i=1}^3 A_i^{(1)} \otimes A_i^{(2)} \) and a CP vector of dimension \(2\) and rank \(2 \), i.e., \( \mathbf{x} = \sum_{j=1}^2 x_j^{(1)} \otimes x_j^{(2)} \).
Then, the application of \( \mathbf{A} \) to \( \mathbf{x} \) yields \( \mathbf{A} \mathbf{x} = (\sum_{i=1}^3 A_i^{(1)} \otimes A_i^{(2)}) ( \sum_{j=1}^2 x_j^{(1)} \otimes x_j^{(2)}) = \sum_{i=1}^3 \sum_{j=1}^2 A_i^{(1)}  x_j^{(1)} \otimes  A_i^{(2)} x_j^{(2)} = \sum_{k=1}^6 y_k^{(1)} \otimes y_k^{(2)}\) with \( y_k^{(\nu)} \coloneqq A_i^{(\nu)}x_j^{(\nu)} \) for \( k = i + 3(j-1) \) and \(\nu = 1, 2\).
Therefore \( \mathbf{Ax} \) is a CP vector of representation rank \( 6 ~(= 2 \cdot 3)\).

The above example shows that we need a truncation of a tensor to lower rank, i.e., an approximation with a tensor of lower rank.
To guarantee the convergence of iterative methods, we have to guarantee that the truncation error is small enough, cf.~\cite{Hackbusch2008}.

The set of CP tensors of rank \(r\) is, however, not closed which makes the approximation of a CP tensor of rank \(r\) an ill-posed problem, cf.~\cite{deSilva2008}. 
Therefore, we cannot guarantee that the truncation error will be small enough to yield convergence of the iterative method.
To overcome this drawback, we use the \emph{hierarchical Tucker} format to represent and compute the solution of a linear system.

%% file: tex/ht_format.tex
The general idea of the hierarchical Tucker format, which was first introduced in~\cite{Hackbusch2009} and further analyzed in~\cite{Grasedyck2010}, is to define a hierarchy among the modes \(\mathcal{D} = \left\lbrace 1, \dots, d \right\rbrace \). 
To do so, we define the so-called \emph{dimension tree} analogously to~\cite[Definition 3.1]{Grasedyck2010}.
\begin{definition}[dimension tree]\label{def:dimensionTree}
A \emph{dimension tree} \(\mathcal{T}\) for dimension \(d \in \N \) is a binary tree with nodes labeled by non-empty subsets of \(\mathcal{D}\). 
Its root is labeled with \(\mathcal{D}\), each leaf node is labeled with a single-element subset \(\treesymbol = \left\lbrace \ell \right\rbrace \subseteq \mathcal{D}\), and each inner node is labeled with the disjoint union of its two children.
We will identify a node with its label \( \treesymbol \) and therefore write \( \treesymbol \in \mathcal{T} \).
\end{definition}
Figure~\ref{fig:dtee} shows an example of a dimension tree for \(d = 4\).
The labels of dimension trees lead to the corresponding \emph{matricization} for each node which we define as in~\cite[Definition \(3.3\)]{Grasedyck2010}:
\begin{definition}[matricization and vectorization]\label{def:matricization}
Let \(\boldsymbol{\phi} \in \mathbb{R}^{n_1 \times \cdots \times n_d}\), \(\treesymbol \subseteq \mathcal{D}\) with \(\treesymbol \neq \emptyset \), and \(g \coloneqq \mathcal{D} \setminus \treesymbol \). 
The \emph{matricization} of \(\boldsymbol{\phi}\) corresponding to \(\treesymbol \) is defined as \(\boldsymbol{\phi}^{(\treesymbol)} \in \mathbb{R}^{n_{\treesymbol} \times n_{g}}\), where \(n_{\treesymbol} \coloneqq \prod_{\ell \in \treesymbol} n_\ell \) and \(n_g \coloneqq \prod_{\ell \in g} n_\ell \), with \( \boldsymbol{\phi}^{(\treesymbol)} [  {(i_j)}_{j \in \treesymbol},{(i_j)}_{j \in g} ] \coloneqq \boldsymbol{\phi} [ i_1,\dots,i_d ] \) for all \( i={(i_j)}_{j \in \mathcal{D}} \).
In particular, \(\boldsymbol{\phi}^{(\mathcal{D})} \in \mathbb{R}^{n_1 \cdots  n_d}\) holds, which can also be interpreted as the \emph{vectorization} of \( \boldsymbol{\phi} \).
\end{definition}
A matricization can be interpreted as an unfolding of the tensor as illustrated in Figure~\ref{fig:matricization}.
\begin{figure}
\centering
\begin{minipage}{0.45\textwidth}
\centering
\Tree[.{\(\boxed{ \left\lbrace 1, 2, 3, 4 \right\rbrace }\)} [.{\(\boxed{ \left\lbrace 1, 2 \right\rbrace }\)} {\(\boxed{ \left\lbrace 1 \right\rbrace }\)} {\(\boxed{ \left\lbrace 2 \right\rbrace }\)} ] [.{\(\boxed{ \left\lbrace  3, 4 \right\rbrace}\)} {\(\boxed{ \left\lbrace 3 \right\rbrace}\)} {\(\boxed{ \left\lbrace 4 \right\rbrace }\)} ] ] 
\caption{Dimension tree for dimension \(d=4\).}\label{fig:dtee}
\end{minipage}
\hfill
\begin{minipage}{0.45\textwidth}
\centering
\begin{tikzpicture}[scale=4.0]
\pgfmathsetmacro{\cubex}{0.4}
\pgfmathsetmacro{\cubey}{0.4}
\pgfmathsetmacro{\cubez}{0.4}
\pgfmathsetmacro{\smallcubex}{0.05}
\pgfmathsetmacro{\smallcubey}{0.05}
\pgfmathsetmacro{\smallcubez}{0.05}
\pgfmathsetmacro{\largecubex}{0.8}
\pgfmathsetmacro{\largecubey}{0.8}
\pgfmathsetmacro{\largecubez}{0.8}

\draw[black] (0,0,0) -- ++(0,0,-\cubez) -- ++(0,-\cubey,0) --
		++(0,0,\cubez) -- cycle;
\draw[black] (0,0,0) -- ++(-\cubex,0,0) -- ++(0,0,-\cubez) --
		++(\cubex,0,0) -- cycle;
		
\draw[black,fill=blue!25] (0,0.15,0.15) -- ++(-\smallcubex,0,0) --
		++(0,-\cubey,0) -- ++(\smallcubex,0,0) -- cycle;
\draw[black,fill=blue!25] (0,0.15,0.15) -- ++(0,0,-\smallcubez) --
		++(0,-\cubey,0) -- ++(0,0,\smallcubez) -- cycle;
\draw[black,fill=blue!25] (0,0.15,0.15) -- ++(-\smallcubex,0,0) --
		++(0,0,-\smallcubez) -- ++(\smallcubex,0,0) -- cycle;
		
\draw[black] (0,0,0) -- ++(-\cubex,0,0) -- ++(0,-\cubey,0) --
		++(\cubex,0,0) -- cycle;
\draw[black,fill=blue!25] (0.7,0.105,0) -- ++(-\smallcubex,0,0) --
		++(0,-\cubey ,0) -- ++(\smallcubex,0,0) -- cycle;
\draw[black] (1.2,0.105,0) -- ++(-\largecubex,0,0) -- ++(0,-\cubey,0) --
		++(\largecubex,0,0) -- cycle;

\node(end) at (0.7,0.105){};
\node(start) at  (0.0, 0.15, 0.15){};
\draw[black, ->] (start)  edge [bend left,  above] (end);
		
\end{tikzpicture}
\caption{Visual representation of a matricization.}\label{fig:matricization}
\end{minipage}
\end{figure}
Based on the concept of matricizations the \emph{hierarchical Tucker rank} is defined accordingly to~\cite[Definition \(3.4\)]{Grasedyck2010}:
\begin{definition}[hierarchical Tucker rank]\label{def:HTRank}
Let \(\boldsymbol{\phi} \in \mathbb{R}^{n_1 \times \cdots \times n_d}\) and \( \mathcal{T} \) be a dimension tree. 
The \emph{hierarchical Tucker rank} of \(\boldsymbol{\phi}\) is defined as \( \operatorname{rank}_{ \mathcal{T} }(\boldsymbol{\phi}) \coloneqq {(r_\treesymbol)}_{\treesymbol \in \mathcal{T}} \), where \( r_{\treesymbol} \coloneqq \operatorname{rank}(\boldsymbol{\phi}^{(\treesymbol)} )\) denotes the matrix rank of the matricization \(\boldsymbol{\phi}^{(\treesymbol)}\) for all \(\treesymbol \in \mathcal{T}\).

The set of tensors with hierarchical Tucker rank node-wise bounded by \({(r_{\treesymbol})}_{\treesymbol \in \mathcal{T}}\) is defined as \( \HT ( \mathcal{T} , {(r_\treesymbol)}_{\treesymbol \in \mathcal{T}} ) \coloneqq 
\{  \boldsymbol{\gamma} \in \mathbb{R}^{n_1 \times \cdots \times n_d}  \vert \operatorname{rank}( \boldsymbol{\gamma}^{(\treesymbol)}) \leq r_\treesymbol \text{ for all } \treesymbol \in \mathcal{T}  \} \).
\end{definition}
Using the dimension tree, the concept of matricization, and the hierarchical Tucker rank, one can define the representation of a tensor within the hierarchical Tucker format, cf.~\cite[Definition \(3.6\)]{Grasedyck2010}.
The memory required for a hierarchical Tucker representation, with dimension tree \( \mathcal{T} \) and representation rank \( {(r_\treesymbol)}_{\treesymbol \in \mathcal{T}} \), of a tensor \(\boldsymbol{\phi} \in \mathbb{R}^{n_1 \times \cdots \times n_d}\) for \(n= \max_{\ell \in \mathcal{D}} n_{\ell}\) and \(r= \max_{\treesymbol \in \mathcal{T}} r_\treesymbol \) is given by \(\mathcal{O} \left( rd n+ r^3 d \right)\), cf.~\cite[Lemma \(3.7\)]{Grasedyck2010}.
The existence of a truncation method of a low\hyp{}rank tensor \(\boldsymbol{\phi} \in \HT (\mathcal{T},\, {(r_\treesymbol)}_{\treesymbol \in \mathcal{T}})\) down to lower rank \({(\tilde{r}_\treesymbol)}_{\treesymbol \in \mathcal{T}}\) with an arithmetic cost in \(\mathcal{O}\left( r^2 d n+ r^4 d \right)\) was proven in~\cite{Grasedyck2010}.
The resulting approximation \(\tilde{\boldsymbol{\phi}} \coloneqq \operatorname{truncate}(\boldsymbol{\phi}) \in \HT ( \mathcal{T} , { ( \tilde{r}_\treesymbol ) }_{ \treesymbol \in \mathcal{T}} )\) fulfills the quasi\hyp{}optimal error estimation
\begin{equation*}
\Vert \boldsymbol{\phi} - \tilde{\boldsymbol{\phi}} \Vert \leq \sqrt{2 d - 3} \inf_{ \boldsymbol{\gamma} \in \HT (\mathcal{T}, {(\tilde{r}_\treesymbol)}_{\treesymbol \in \mathcal{T}})} \Vert \boldsymbol{\phi} -  \boldsymbol{\gamma}  \Vert .
\end{equation*}
Further, we can transfer a CP representation of a tensor vector or tensor operator with CP rank \(r\) into a hierarchical Tucker representation with rank node-wise bounded by \(r\), cf.~\cite[Theorem \(11.17\)]{Hackbusch2012}. 
Following this approach, we represent the operator and the \righthandside\ in the hierarchical Tucker format.
For solving parameter-dependent linear problems in the hierarchical Tucker format, we use the preconditioned conjugate gradients~(PCG) method.
In Algorithm~\ref{alg:pcg} the PCG method is briefly introduced similar to~\cite[Algorithm 2]{Kressner2011}.

\begin{algorithm}
\caption{preconditioned conjugate gradients method with truncation.}\label{alg:pcg}
\begin{algorithmic}[1]
\REQUIRE{CP operator \(\AlgA \), CP vector \(\Algb \), CP rank~\(1\) preconditioner \(\AlgM \), initial guess \(\Algphi_{(0)}\) in the hierarchical Tucker format}
\ENSURE{Approximate solution \(\Algphi \) in the hierarchical Tucker format of \(\AlgA \Algphi = \Algb \) }
\STATE{ \( \AlgR_{(0)} = \operatorname{truncate}\left( \Algb - \AlgA  \Algphi_{(0)} \right)\) }
\STATE{ \(\AlgZ_{(0)} =  \AlgM^{-1} \AlgR_{(0)} \) }
\STATE{ \(\AlgP_{(0)} = \AlgZ_{(0)}\) }
\STATE{ \(\AlgQ_{(0)} = \operatorname{truncate}\left( \AlgA \AlgP_{(0)}\right) \) }
\STATE{ \(k = 0\) }
\WHILE{\(\frac{\left\lVert \AlgR_{(k)} \right\rVert }{\left\lVert \Algb \right\rVert} > \varepsilon \) \AND{}\(k < k_{\max}\)}
\STATE{\(\Algphi_{(k+1)} = \operatorname{truncate}\left( \Algphi_{(k)} + \frac{\left\langle \AlgR_{(k)}, \AlgP_{(k)} \right\rangle}{\left\langle \AlgQ_{(k)}, \AlgP_{(k)} \right\rangle} \AlgP_{(k)} \right)\) }
\STATE{ \(\AlgR_{(k+1)} = \operatorname{truncate}\left( \Algb - \AlgA  \Algphi_{(k+1)} \right)\) }
\STATE{ \(\AlgZ_{(k+1)} =  \AlgM^{-1} \AlgR_{(k+1)} \) }
\STATE{ \(\AlgP_{(k+1)} = \operatorname{truncate}\left( \AlgZ_{(k+1)} - \frac{ \left\langle \AlgQ_{(k)}, \AlgZ_{(k+1)} \right\rangle}{\left\langle \AlgQ_{(k)}, \AlgP_{(k)} \right\rangle } \AlgP_{(k)} \right) \) }
\STATE{ \(\AlgQ_{(k+1)} = \operatorname{truncate}\left( \AlgA \AlgP_{(k+1)} \right) \) }
\STATE{ \(k = k + 1\) }
\ENDWHILE{}
\end{algorithmic} 
\end{algorithm}

The PCG method in Algorithm~\ref{alg:pcg} approximates the solution of a parameter-dependent linear system numerically within the hierarchical Tucker format if the tensor operator \(\mathbf{A}\) is positive definite and symmetric.
In~\cite[Lemma 5]{Grasedyck2017} the authors proved that this algorithm converges if the truncation error \(\varepsilon \) is small enough.
Algorithm~\ref{alg:pcg} comprises additions and inner products of two tensors in hierarchical Tucker format which have an arithmetic cost in \( \mathcal{O}(dnr^2 + dr^4) \), application of an operator which has an arithmetic cost in \( \mathcal{O}( dn^2 r) \), and evaluation of an entry of the represented tensor which has an arithmetic cost in \( \mathcal{O}(dr^3) \).
Hence, for small rank \(r\) most of the operations needed for the PCG method scale linearly in the dimension \(d\) and the mode size \(n\), thus yielding an efficient method to solve parameter-dependent linear systems using low-rank tensor formats.

This means that, if we are able to prove the existence of a low-rank representation of the operator and right-hand side of the forward EMG problem, we can compute the solution of the linear system data-sparse and fast within the hierarchical Tucker format.

Finding conditions that guarantee the existence of a \lowrank\ approximation for a given tensor is a research topic of its own~\cite{Bachmayr2017,Dahmen2016,Kressner2016}.
This goes beyond the scope of this article, and we thus assume that the solution of the parameter-dependent EMG forward problem has a low-rank approximation.
This is backed up by the numerical experiments in Section~\ref{sec:NumExp}.

%% file: tex/Bayes_inverse_EMG.tex
\section{The Bayesian inverse EMG problem}\label{sec:Bayes_inverse}
We present our first main contribution: The proof of the well-posedness of the Bayesian inverse EMG problem discussed in Section~\ref{sec:EMG_model}.
Note that the proof of the well-posedness is valid for any bounded conductivity \(\sigma_\text{i}\) that can be represented through parameters \(p \in \mathcal{J}\) for any parameter space \(\mathcal{J}\). 
For diagonal conductivities these parameters are the diagonal entries of \(\sigma_\text{i}\) and \(\mathcal{J} = {[s_{-},s_{+}]}^3\). 
In the more general case of space-dependent intracellular conductivities, the parameters can be chosen as the coefficients of a Karhun-Loève expansion of \(\sigma_\text{i}(x,\omega)\), see, e.g.,~\cite{R2,R3}. 
The following proof thus holds for both, space-independent and space-dependent, conductivities.

First we prove the existence of the posterior distribution \(\rho^\text{EMG}\) of parameters \( p \) given measurements \(\phi_\text{EMG}^\text{meas}\) for a prior \(\rho \) using the infinite-dimensional version of Bayes' theorem for inverse problems~\cite[Theorem 3.4]{R2}.
\begin{theorem}[Bayes' theorem for our inverse EMG problem]\label{theo:applicability_Bayes}
Let 
\(\mathbb{Q}_0\) and \(\mathbb{Q}_p\) denote the measures with distribution \( \mathcal{N}(0,\Xi)\) and \(\mathcal{N}(\mathcal{G}_{\mathbf{x}}(p),\Xi)\). 
Then,
\begin{enumerate}[label=B.\arabic*]
\item\label{Bayes_theo_1} the scaling factor \(Z \coloneqq \int_{\mathcal{J}} \exp\big( -\Phi(p;\phi_\text{EMG}^\text{meas}) \big)\rho(\mathrm{d}p)\) is positive \(\mathbb{Q}_0\)-almost surely,
\item\label{Bayes_theo_2} the potential \(\Phi : \mathcal{J} \times \mathbb{R}^{M} \to \mathbb{R} \), as defined in~\eqref{eq:def_potential}, is \(\nu_{0}\)-measurable with product measure \(\nu_0(\mathrm{d}p,\mathrm{d}\phi) \coloneqq  \rho(\mathrm{d}p)\mathbb{Q}_0(\mathrm{d}\phi) \),
\item\label{Bayes_theo_3} for \(\phi_\text{EMG}^\text{meas}\) the conditional distribution \(\rho^\text{EMG}\) exists, \(\rho^\text{EMG}\) is absolutely continuous with respect to \(\rho \), and
\begin{equation*}
\frac{\mathrm{d}\rho^\text{EMG}}{\mathrm{d}\rho}(p) = \frac{1}{Z}\exp\big( -\Phi(p;\phi_\text{EMG}^\text{meas}) \big)
\end{equation*}
\(\nu \)-almost surely with the product measure \(\nu(\mathrm{d}p,\mathrm{d}\phi) \coloneqq  \rho(\mathrm{d}p)\mathbb{Q}_p(\mathrm{d}\phi)\).
\end{enumerate}
\end{theorem}
To prove the above theorem, we need the boundedness and Lipschitz continuity of the observation operator as stated in the following lemma:
\begin{lemma}\label{lemma:wellposed_forward}
The observation operator is bounded and Lipschitz continuous with respect to \(p\), i.e., there exist constants \(0 < C,L_p < \infty \) such that
\begin{align}
	\norm{\mathcal{G}_\mathbf{x}(p)}_{\mathbb{R}^M} &\le C \label{eq:G_bounded}\\
	\norm{\mathcal{G}_\mathbf{x}(p_1) - \mathcal{G}_\mathbf{x}(p_2)}_{\mathbb{R}^M} &\le L_p \norm{p_1 - p_2}_\infty \label{eq:G_Lipschitz}
\end{align}
for all \(p,p_1,p_2 \in \mathcal{J}\).
\end{lemma}
The proof consists of basic calculations and estimations on the weak form of the deterministic EMG forward problem and is thus left to the reader.

\input{./tex/proof_Bayes_theorem}
The well-posedness of the Bayesian inverse EMG problem also includes the continuity of the posterior \(\rho^\text{EMG}\) with respect to the data \(\phi_\text{EMG}^\text{meas}\). 
Therefore, we need to define a metric on the space of measures.
Similar to~\cite{R2,R3} we choose the \emph{Hellinger metric}.
\begin{definition}[Hellinger metric]
Let \(\mu_1\) and \(\mu_2\) denote two probability measures that are absolutely continuous with respect to a probability measure \(\zeta \). 
The \emph{Hellinger metric} of \(\mu_1\) and \(\mu_2\) is then defined as
\begin{equation*}
     d_\text{Hell}(\mu_1,\mu_2) \coloneqq {\Big( \frac12 \int {\Big( \sqrt{\tfrac{\mathrm{d}\mu_1}{\mathrm{d}\zeta}} - \sqrt{\tfrac{\mathrm{d}\mu_2}{\mathrm{d}\zeta}} \Big)}^2 \,\mathrm{d}\zeta \Big)}^{\frac{1}{2}}.
\end{equation*}
\end{definition}

With the help of the Hellinger metric we now prove the Lipschitz continuity of the posterior \(\rho^\text{EMG}\) with respect to measured EMG data.

\begin{theorem}\label{theo:wellposed_Bayes}
Let
\(\rho^\text{EMG}\) denote the solution of our Bayesian inverse EMG problem given by Theorem~\ref{theo:applicability_Bayes}. 
Then \(\rho^\text{EMG}\) depends Lipschitz continuously on the measured data \(\phi_\text{EMG}^\text{meas}\) with respect to the Hellinger metric. 
This means there exists a positive constant \(L > 0\) such that
\begin{equation}\label{eq:well_posed_hellinger_norm}
   d_\text{Hell}(\rho^\text{EMG}_1,\rho^\text{EMG}_2) \le L \norm{\phi_1 - \phi_2}_\Xi
\end{equation}
holds for all \(\phi_1,\phi_2 \in\mathbb{R}^{M}\) and the posterior distributions \(\rho_1^\text{EMG}\) and \(\rho_2^\text{EMG}\) of \(\sigma \) given \(\phi_1 \) and \(\phi_2 \).
\end{theorem}
To prove the above theorem, we need the following lemma:
\begin{lemma}\label{lemma:Z_Lipschitz}
The scaling factor \(Z(\phi) = \int_{\mathcal{J}} \exp(-\Phi(p,\phi))\,\mathrm{d}\rho(p)\) is Lipschitz continuous in \(\phi \), i.e., there exists a constant \(L_\text{Z} > 0\) such that
  \begin{align}
     |Z(\phi_1) - Z(\phi_2)| \le L_\text{Z} \norm{\phi_1 - \phi_2}_\Xi\label{eq:Z_Lipschitz}
  \end{align}
holds for all \(\phi_1,\phi_2 \in \mathbb{R}^{M}\) with \( \phi_1 \ne \phi_2 \).
\end{lemma}
The statement follows from Lemma~\ref{lemma:wellposed_forward} with basic calculations and estimations and is thus left to the reader.

\input{./tex/proof_well-posed_Bayes}
\begin{remark}
The estimate in~\eqref{eq:well_posed_hellinger_norm} also describes the behavior of the posterior with respect to the discretization of the underlying equations.
\end{remark}

Recapitulating Theorems~\ref{theo:applicability_Bayes} and~\ref{theo:wellposed_Bayes} shows that modeling the measurement error as a stochastic quantity leads to a regularization of our inverse EMG problem, see also~\cite{R2}.

%% file: tex/proof_Bayes_theorem.tex
\begin{proof}[Proof of Theorem~\ref{theo:applicability_Bayes}]
The proof is based on the proof of the measurability of the potential \(\Phi \). 
Since~\ref{Bayes_theo_1} and~\ref{Bayes_theo_2} are the assumptions required for the Bayes Theorem in~\cite[Theorem 3.4]{R2} to hold,~\ref{Bayes_theo_3} follows directly once~\ref{Bayes_theo_1} and~\ref{Bayes_theo_2} are proven. 
As the \(\nu_0\)-measurability of \(\Phi \), meaning that \(\Phi \) is \(\rho \)-measurable in \(p\) and \(\mathbb{Q}_0\)-measurable in \(\phi_\text{EMG}^\text{meas}\), follows from the Lipschitz continuity of the corresponding mappings, we show that
\begin{enumerate}
\item \(\Phi \) is Lipschitz continuous with respect to \(p\) and
\item \(\Phi \) is Lipschitz continuous with respect to \(\phi_\text{EMG}^\text{meas}\).
\end{enumerate}
Note that we also need the Lipschitz continuity of \(\Phi \) to prove that the posterior depends continuously on the measurement data in Theorem~\ref{theo:wellposed_Bayes}.
For ease of notations, we introduce the shorthand \(\langle u,v \rangle_\Xi \coloneqq \langle\Xi^{-\frac12}u,\Xi^{-\frac12}v\rangle \) for \(u,v \in \mathbb{R}^{M}\) and neglect the second argument of the potential \(\Phi \).

\begin{enumerate}
\item Let \(p_1,p_2 \in \mathcal{J}\) with \(p_1 \ne p_2\), and (TI) and (HI) denote the triangle and Hölder's inequality. 
Using Lemma~\ref{lemma:wellposed_forward}, we have
\begin{align*}
          &\, |\Phi(p_1) - \Phi(p_2)| \\
        =&\,  \frac12 \left| \langle\mathcal{G}(p_1),\mathcal{G}(p_1)\rangle_\Xi - \langle\mathcal{G}(p_2),\mathcal{G}(p_2)\rangle_\Xi + 2\langle\phi_\text{EMG}^\text{meas},\mathcal{G}(p_2) - \mathcal{G}(p_1)\rangle_\Xi \right|\\
         \underset{\mathclap{\text{(HI)}}}{\overset{\mathclap{\text{(TI)}}}{\le}}&\, \frac12 \left( \norm{\mathcal{G}(p_1)}_\Xi \norm{\mathcal{G}(p_1) - \mathcal{G}(p_2)}_\Xi + \norm{\mathcal{G}(p_1) - \mathcal{G}(p_2)}_\Xi \norm{\mathcal{G}(p_2)}_\Xi \right)\\
          &+ \norm{\phi_\text{EMG}^\text{meas}}_\Xi \norm{\mathcal{G}(p_2) - \mathcal{G}(p_1)}_\Xi \\
          \overset{\mathclap{\eqref{eq:G_bounded}}}{\le}&\, C \norm{\mathcal{G}(p_2) - \mathcal{G}(p_1)}_{\mathbb{R}^{M}}   \overset{\mathclap{\eqref{eq:G_Lipschitz}}}{\le}\, CL_p \norm{p_1 - p_2}_{\infty} .
\end{align*}
\item For \(\phi_1,\phi_2 \in\mathbb{R}^{M}\) with \(\phi_1 \ne \phi_2\) we express the norms in the definition of \(\Phi \) as scalar products obtaining
\begin{align}
         | \Phi(p,\phi_{1}) - \Phi(p,{\phi}_{2}) | &= \frac12 \left| \norm{ \phi_{1} - \mathcal{G}(p) }_\Xi^2 - \norm{\phi_{1}}_\Xi^2 - \norm{ \phi_{2} - \mathcal{G}(p) }_\Xi^2 + \norm{\phi_{2}}_\Xi^2 \right|\nonumber \\
         &= \left| \langle (\phi_2 - \phi_1),\mathcal{G}(p) \rangle_\Xi \right| \overset{\mathclap{\text{(HI)}}}{\le} \norm{\phi_2 - \phi_1}_\Xi \norm{\mathcal{G}(p)}_\Xi\nonumber \\
         &\overset{\mathclap{\eqref{eq:G_bounded}}}{\le} \underset{\eqqcolon L_\phi}{\underbrace{C \norm{\Xi}_\infty^2}} \norm{\phi_1 - \phi_2}_{\mathbb{R}^M} . \label{eq:potential_Lipschitz_in_data}
\end{align}
\end{enumerate}
This concludes the proof.
\end{proof}

%% file: tex/proof_well-posed_Bayes.tex
\begin{proof}[Proof of Theorem~\ref{theo:wellposed_Bayes}]
Let \(\rho^\text{EMG}_1,\rho^\text{EMG}_2\) denote the solutions of the Bayesian inverse EMG problem for given measurements \(\phi_1 \ne \phi_2\). 
For simplicity, we write \(\Phi_j \coloneqq \Phi(p,\phi_j)\) and \(Z_j \coloneqq Z(\phi_j)\), \(j=1,2\). 
We estimate the Hellinger distance between the two posterior distributions using Young's inequality~(YI), the Lipschitz continuity of the exponential function and the inverse of the square root on bounded domains with constants \(L_\text{e}\) and \(L_\text{sqrt} \) and Lemma~\ref{lemma:Z_Lipschitz}:
\begin{align*}
   2 d_\text{Hell}{(\rho^\text{EMG}_1,\rho^\text{EMG}_2)}^2 &= \; \int_{\mathcal{J}} {\left[ {\Big( \frac{1}{Z_1}\exp(-\Phi_1) \Big)}^\frac12 - {\Big( \frac{1}{Z_2}\exp(-\Phi_2) \Big)}^\frac12 \right]}^2 \, \mathrm{d}\rho(p)\\
   &\underset{\mathclap{\exp,{({\,})}^{-\frac{1}{2}}\text{Lip.}}}{\overset{\mathclap{\text{(YI)}}}{\le}} \; 2\int_{\mathcal{J}} \frac{1}{Z_1} L_\text{e}^2 \left| \Phi_1 - \Phi_2 \right|^2\,\mathrm{d}\rho(p) \\
   &\;\quad+ 2\int_{\mathcal{J}} L_\text{sqrt}^2 \left| Z_1 - Z_2 \right|^2 \exp(-\Phi_2)\,\mathrm{d}\rho(p)\\
   &\overset{\mathclap{\eqref{eq:Z_Lipschitz}}}{\le} \; 2\int_{\mathcal{J}} \frac{1}{Z_1} L_\text{e}^2 L_\phi^2 \norm{\phi_1 - \phi_2}_\Xi^2\,\mathrm{d}\rho(p) \\
   &\;\quad+ 2\int_{\mathcal{J}} L_\text{sqrt}^2 L_\text{Z}^2 \norm{\phi_1 - \phi_2}_\Xi^2 \exp(-\Phi_2)\,\mathrm{d}\rho(p)\\
   &= \; 2\left( L_\text{e}^2 L_\phi^2 \frac{1}{Z_1} + L_\text{sqrt}^2 L_\text{Z}^2 Z_2 \right) \norm{\phi_1 - \phi_2}_\Xi^2 .
\end{align*}
As \(Z_1 > 0\) holds, it follows that \(\frac{1}{Z_1} < \infty \). 
It thus remains to prove that \(Z_2 < \infty \) which is a consequence of \(\mathcal{G} \) being bounded and \(\rho(\mathcal{J}) = 1\). 
The assertion follows with Lipschitz constant \(L_\rho^2 \coloneqq L_\text{e}^2 L_\phi^2 \frac{1}{Z_1} + L_\text{sqrt}^2 L_\text{Z}^2 Z_2\).
\end{proof}

%% file: tex/linear_system.tex
\section{Discretization and tensorization}\label{sec:guiding_example}

As described in Section~\ref{sec:EMG_model}, we compute the posterior distribution \(\rho^\text{EMG}\) using a Metropolis-Hastings algorithm.
We obtain an approximation of the posterior by drawing a finite number of samples. 
Additionally, we discretize the forward operator \( \mathcal{G}_\mathbf{x} \) as follows.
In accordance with Section~\ref{sec:Bayes_inverse}, we show a discretization for the more general case of space-dependent intracellular conductivities and mention that this discretization simplifies slightly for the space-independent case.

%% file: tex/finite_difference_operator.tex
With \(x = {(x_1,x_2,x_3)}^\top \in D_\text{M}\) the left-hand side of equation~\eqref{eq:main_eq} reads
\begin{equation}\label{eq:cookie3D} 
	\begin{aligned}
	A \phi_\text{e}  \coloneqq \nabla \cdot \left( ( \sigma_\text{i}(x) + \sigma_\text{e} ) \nabla \phi_\text{e}(x) \right)  =  \sum_{j=1}^3 \frac{\partial}{\partial x_j} \left( ( \sigma_\text{i}(x) + \sigma_\text{e} ) \frac{\partial}{\partial x_j} \phi_\text{e}(x) \right)
	\end{aligned}
\end{equation}
and the right-hand side is given by
\begin{equation}\label{eq:cookie3DRHS} 
	\begin{aligned}
	b  \coloneqq  - \nabla \cdot \left( \sigma_\text{i}(x)  \nabla V_\text{m}(x) \right)  =  \sum_{j=1}^3 -\frac{\partial}{\partial x_j} \left(  \sigma_\text{i}(x)  \frac{\partial}{\partial x_j} V_\text{m}(x) \right).
	\end{aligned}
\end{equation}
Since our forward solver uses a finite difference discretization, we consider the same discretization using centered differences of second order, and therefore assume that \(\phi_\text{e} \in C^4(D_\text{M})\) and \(\sigma_\text{i} \in C^1(D_\text{M})\). 
This is reasonable under our assumptions. 
Our theoretical and numerical results directly generalize to, e.g., finite element discretizations of arbitrary but given muscle geometries. 
The practical realization is future work.

In the following we use \(h = (h_\text{M},h_\text{t},h_\sigma)\) to indicate the discretization of the muscle geometry by \(h_\text{M}\), the time by \(h_\text{t}\) and the parameter space by \(h_\sigma \).
We denote the grid points by \((x_{j_1}\), \(x_{j_2}\), \(x_{j_3})\), \(j_k = 0,\dots, n \), for \(n \in \N \) and a discrete conductivity at grid point \((x_{{j_1}}\), \(x_{{j_2}}\), \(x_{{j_3}})\) by \(\sigma_{j_1, j_2, j_3}\).
\begin{theorem}\label{th:cookie3DStencil}
For
\begin{equation}\label{eq:cookie3dOhneSigmaE}
\begin{aligned} 
B \phi  \coloneqq \nabla \cdot \left( \sigma(x)  \nabla \phi(x) \right) =  \sum_{j=1}^3 \frac{\partial}{\partial x_j} \left(  \sigma(x)  \frac{\partial}{\partial x_j} \phi(x) \right)
\end{aligned}
\end{equation}
a second-order consistent stencil is given by 
\begin{align*}
&
\begin{bmatrix}
0 & 0 & 0 \\
0 &  \frac{ \sigma_{j,j,j-1} + \sigma_{j,j,j} }{2h_\text{M}^2} & 0 \\
0 & 0 & 0
\end{bmatrix} \text{ in the first plane, in the second plane by} \\
& 
\begin{bmatrix}
0 & \frac{\sigma_{j,j-1,j} + \sigma_{j,j,j}}{2h_\text{M}^2} & 0 \\
 \frac{\sigma_{j-1,j,j} + \sigma_{j,j,j}}{2h_\text{M}^2} & -\frac{ \sigma_{j-1,j,j} + \sigma_{j,j-1,j} +\sigma_{j,j,j-1} + 6 \sigma_{j,j,j} + \sigma_{j,j,j+1} + \sigma_{j,j+1,j} + \sigma_{j+1,j,j} }{2h_\text{M}^2} & \frac{ \sigma_{j,j,j} + \sigma_{j+1,j,j} }{2h_\text{M}^2} \\
0 &  \frac{ \sigma_{j,j,j} + \sigma_{j,j+1,j} }{2h_\text{M}^2} & 0
\end{bmatrix}\\
& \text{and }  
\begin{bmatrix}
0 & 0 & 0 \\
0 &  \frac{ \sigma_{j,j,j} + \sigma_{j,j,j+1} }{2h_\text{M}^2} & 0 \\
0 & 0 & 0
\end{bmatrix} \text{ in the third plane.}
\end{align*}
\end{theorem}
\begin{proof}
Because of the Kronecker product structure of~\eqref{eq:cookie3dOhneSigmaE} the statement follows from the one-dimensional case.
There, Taylor's theorem and equating the coefficients of
\begin{align*}
{(B \phi )}_j & = \left( \sigma'(x_j) \phi'_j + \sigma(x_j) \phi''_j \right) \quad \text{ and }\\
{(B_h \phi_h)}_j & = \frac{1}{h_\text{M}^2} \left( -\tilde{\sigma}_{j} \phi_{j-1} + (\tilde{\sigma}_{j} + \tilde{\sigma}_{j+1}) \phi_{j} - \tilde{\sigma}_{j+1} \phi_{j+1} \right),
\end{align*}
yields \(\tilde{\sigma}_j = \frac{\sigma_{j-1}+\sigma_{j}}{2}\) for a second-order consistent stencil given by  
\begin{equation*}
	\frac{1}{h_\text{M}^2}
		\begin{bmatrix}
			 \frac{\sigma_{j-1} + \sigma_{j}}{2} & - \frac{ \sigma_{j-1} + 2 \sigma_{j} + \sigma_{j+1} }{2} & \frac{ 	\sigma_{j} + \sigma_{j+1} }{2}
		\end{bmatrix},
\end{equation*}
immediately finishing the proof.
\end{proof}
Next, we derive an affine representation of the discrete operator and prove a low-rank tensor format representation of the operator and the right-hand side of the forward EMG problem.
This is our second main contribution. 
\begin{corollary}\label{corol:aff3D}
An affine representation of the discrete operator in the three-dimensional case is given by
\begin{multline*}
\frac{\sigma_{j,j-1,j}}{h_\text{M}^2} M_{j,j-1,j} + \frac{\sigma_{j,j,j+1}}{h_\text{M}^2} M_{j,j,j+1}	 + \frac{\sigma_{j-1,j,j}}{h_\text{M}^2} M_{j-1,j,j} \\
 +  \frac{\sigma_{j,j,j}}{h_\text{M}^2} M_{j,j,j}  + \frac{\sigma_{j+1,j,j}}{h_\text{M}^2} M_{j+1,j,j} + \frac{\sigma_{j,j,j-1}}{h_\text{M}^2} M_{j,j,j-1} + \frac{\sigma_{j,j+1,j}}{h_\text{M}^2} M_{j,j+1,j} ,
\end{multline*}
where in the first plane the stencil is given by
\begin{align*}
M_{j,j-1,j}^{(:,:,1)} & =
M_{j,j,j+1}^{(:,:,1)} =
M_{j-1,j,j}^{(:,:,1)} = 
M_{j,j,j}^{(:,:,1)} = 
		M_{j+1,j,j}^{(:,:,1)} = 
		\begin{bmatrix}
			0 & 0 & 0 \\
			0 & 0 & 0 \\
			0 &  0 & 0
		\end{bmatrix}, \\
	M_{j,j,j-1}^{(:,:,1)} & =
	\begin{bmatrix}
	0 & 0 & 0 \\
	0 & \frac{1 }{2} &0 \\
	0 &  0 & 0
	\end{bmatrix},
 M_{j,j-1,j}^{(:,:,1)} =
	\begin{bmatrix}
	0 & 0 & 0 \\
	0 & 0 & 0 \\
	0 &  0 & 0
	\end{bmatrix},
\end{align*}
in the second plane by
\begin{align*}
M_{j,j-1,j}^{(:,:,2)} & =
		\begin{bmatrix}
		0 &  \frac{1}{2} & 0 \\
		0 & -\frac{ 1 }{2} & 0 \\
		0 & 0 & 0
		\end{bmatrix},
M_{j,j,j+1}^{(:,:,2)} =
		\begin{bmatrix}
		0 & 0 & 0 \\
		0 &  -\frac{ 1 }{2} & 0 \\
		0 & 0 & 0
		\end{bmatrix},\\
	 M_{j-1,j,j}^{(:,:,2)} & = 
		\begin{bmatrix}
			0 & 0 & 0 \\
			 \frac{1}{2} & -\frac{1}{2} & 0 \\
			0 &  0 & 0
		\end{bmatrix},
M_{j,j,j}^{(:,:,2)}  = 
		\begin{bmatrix}
			0 &  \frac{1}{2} & 0 \\
			\frac{1}{2} & -3 & \frac{ 1 }{2} \\
			0 &  \frac{1 }{2} & 0
		\end{bmatrix}, \\ 
		M_{j+1,j,j}^{(:,:,2)} & = 
		\begin{bmatrix}
			0 & 0 & 0 \\
			0 & -\frac{1}{2} & \frac{ 1 }{2} \\
			0 &  0 & 0
		\end{bmatrix}, 
	M_{j,j,j-1}^{(:,:,2)}  =
	\begin{bmatrix}
	0 & 0 & 0 \\
	0 &  -\frac{1 }{2} &0 \\
	0 &  0 & 0
	\end{bmatrix},
	 M_{j,j-1,j}^{(:,:,2)} =
	\begin{bmatrix}
	0 & 0 & 0 \\
	0 & -\frac{1 }{2} & 0 \\
	0 &  \frac{ 1 }{2} & 0
	\end{bmatrix},
\end{align*}
and in the third plane by
\begin{align*}
	 M_{j,j-1,j}^{(:,:,3)} &=
		\begin{bmatrix}
		0 & 0 & 0 \\
		0 & 0 & 0 \\
		0 & 0 & 0
		\end{bmatrix},
	M_{j,j,j+1}^{(:,:,3)} =
		\begin{bmatrix}
		0 & 0 & 0 \\
		0 &  \frac{ 1 }{2} & 0 \\
		0 & 0 & 0
		\end{bmatrix},\\
	 M_{j-1,j,j}^{(:,:,3)} &= 
		M_{j,j,j}^{(:,:,3)} = 
		M_{j+1,j,j}^{(:,:,3)} = 
	M_{j,j,j-1}^{(:,:,3)} =
	 M_{j,j-1,j}^{(:,:,3)} =
	\begin{bmatrix}
	0 & 0 & 0 \\
	0 & 0 & 0 \\
	0 &  0 & 0
	\end{bmatrix}	.
\end{align*}
\end{corollary} 
\begin{proof}
Follows from Theorem~\ref{th:cookie3DStencil} with linearity.
\end{proof}
We define \(A_h^{(0)} \coloneqq A_{h,\sigma_\text{e}}\) denoting the discrete operator given by Theorem~\ref{th:cookie3DStencil} for constant \(\sigma_\text{e} \in \mathbb{R}^{3 \times 3}\) and~\(A_{h, j_1, j_2, j_3}\) denoting the discrete operator given by the stencil \(M_{j_1, j_2, j_3}\) from Corollary~\ref{corol:aff3D}.
Then the discrete operator of~\eqref{eq:cookie3D} is given by
\begin{equation*}
	A_h \coloneqq A_{h,\sigma_\text{e}} + \sum_{j_1=1}^{m_1} \sum_{j_2=1}^{m_2} \sum_{j_3=1}^{m_3} \sigma_{j_1,j_2,j_3} A_{h,j_1,j_2,j_3}.
\end{equation*}
Using the vectorizations \(\operatorname{vec}(A_{h, j_1, j_2, j_3}) \eqqcolon A_h^{(k)}\) and \(\operatorname{vec}(\sigma_{j_1, j_2, j_3}) \eqqcolon p^{(k)}\), see Definition~\ref{def:matricization}, yields a parameter\hyp{}dependent affine structure of the form
\begin{equation*}
	A_h(p) \coloneqq A_h^{(0)} + \sum_{k=1}^d p^{(k)} A_h^{(k)}
\end{equation*}
with \(p \coloneqq (p^{(1)}, \dots, p^{(d)} )\), where each \(A_h^{(k)}\) is constant, i.e., \(A_h^{(k)}\) is parameter\hyp{}independent.

%% file: tex/finite_difference_rhs.tex
We now take a closer look at the right-hand side and discretize the time variable \(t\) in~\eqref{eq:Rosenfalck} using equidistant time steps \(t_j = jh_\text{t}\), \(j = 0,\ldots, t_{\max}\) for time step size \(h_\text{t}\). 
Multiplying this with the AP velocities \(u_k\), \(k = 1,\ldots,N_\text{MF}\), we achieve \(s_j = u_k t_j\) for the discretization of the muscle fiber coordinate \(s\).

Furthermore, we remark that the linear dependency of the right-hand side on the intracellular conductivity is obvious under our assumptions, which include that the muscle fiber direction is one of the standard unit vectors, i.e., \(\vec{d} = \vec{e}_j\), \(j = 1,2\), or \(3\).
If \(V_\text{m}\) is independent of \(\sigma_\text{i}\), the structure of the right-hand side is the same as the structure of the operator.
Then we see the linear structure of~\eqref{eq:cookie3DRHS} that has the form
\begin{equation*}
b_h(p) \coloneqq \sum_{k = 1}^d p^{(k)} b_h^{(k)}.
\end{equation*}
How to represent an arbitrary right-hand side in a parameter-dependent way is ongoing research.

%% file: tex/parameter_dependent_problems.tex
We now discretize the parameter space by choosing a finite number of parameters~\(p_h \coloneqq (p_h^{(1)}, \dots, p_h^{(\ell)}, \dots, p_h^{(d)})\) from a discrete set~\(\mathcal{J}_h \).
We fix discrete values for all \(p_h^{(\ell)}\), i.e., \(p_h^{(\ell)} \in \lbrace p_h^{(\ell)}(1), p_h^{(\ell)}(2), \dots, p_h^{(\ell)}(n_{\ell}) \rbrace \), and reformulate our problem as:
\begin{equation}\label{eq:problem}
\text{Solve } A_h(p_h) \phi_h(p_h, t) =  b_h(p_h, t) \text{ for all } p_h \in \mathcal{J}_h .
\end{equation}
Assuming that each parameter \(p_h^{(\ell)} \) can take \( n_\ell \) different values, applying classical methods one has to solve a system of \(\prod_{\ell=1}^d n_{\ell} \approx n^d\) linear equations. 
To overcome the curse of dimensionality in this case, we exploit the structure of the linear system, see Section~\ref{sec:repParameterProblems}.
We find a data-sparse representation of the problem that allows us to solve the parameter-dependent system for all \(p_h \in \mathcal{J}_h \) simultaneously, analogously to~\cite{Grasedyck2020}.

%% file: tex/operator.tex
For computing the solution of~\eqref{eq:problem} for all possible \(p_h \in \mathcal{J}_h \), we define a large block-diagonal system with the operator
\begin{align*}
	\mathbf{A} & \coloneqq
	\begin{pmatrix}
		A_{1} ^{(0)}& 0 & \hdots & 0 \\
		0 & 	A_{2} ^{(0)} & \ddots & \vdots \\
		\vdots & \ddots & \ddots & 0 \\
		0 & \hdots  & 0 & 	A_{n} ^{(0)}
	\end{pmatrix} 
	\eqqcolon \operatorname{blkdiag}\left({A}_{1} ^{(0)}, \dots, {A}_{n} ^{(0)} \right),
\end{align*}
where \(	{A}_{j} ^{(0)} =  A_h^{(0)}+ \sum_{\ell = 1}^{d} p_h^{(\ell)}(j) A_h^{(\ell)} \) denotes the \(j-\)th diagonal block.

The memory requirement to store \(\mathbf{A}\), however, grows exponentially in \(n\) and thus, even for moderate values of \(d\) and \(n_{\ell}\), a classical representation of our problem is infeasible. 
Therefore, we reformulate the problem using the notation \({A}_{j}^{(m)} =  \sum_{\ell = m}^{d} p_h^{(\ell)}(j) A_h^{(\ell)} \), \(m = 1,\dots, d \), and \(\Id_{n_k} \) denoting the identity in \( \mathbb{R}^{n_k \times n_k}\), and achieve:
\begin{align*}
		 \mathbf{A} &= \operatorname{blkdiag}\left(A_h^{(0)} + {A}_{1} ^{(1)}, A_h^{(0)}+ {A}_{2} ^{(1)}, \dots, A_h^{(0)}+ {A}_{n} ^{(1)} \right) \\
		&= \operatorname{blkdiag}\left( A_h^{(0)}, A_h^{(0)}, \dots, A_h^{(0)} \right)\\
		& \quad + \operatorname{blkdiag}\left( p_h^{(1)}(1) A_h^{(1)} , p_h^{(1)}(2) A_h^{(1)} , \dots ,  p_h^{(1)}(n_1) A_h^{(1)} \right)\\
		& \quad + \operatorname{blkdiag}\left( {A}_{1}^{(2)} , {A}_{2}^{(2)} , \dots, {A}_{n}^{(2)} \right)\\
		& =  \Id_{n_{d}} \otimes \dots \otimes \Id_{n_{2}} \otimes \Id_{n_{1}} \otimes A^{(0)} \\
		& \quad + \Id_{n_{d}} \otimes \dots \otimes \Id_{n_{2}} \otimes \operatorname{diag}\left(p_h^{(1)} \right) \otimes A_h^{(1)} \\
		& \quad + \cdots  + \operatorname{diag}\left(p_h^{(d)} \right)  \otimes \dots \otimes \Id_{n_{2}} \otimes \Id_{n_{1}} \otimes A_h^{(d)} .
\end{align*}
This leads to the following data-sparse CP representation of the operator
\begin{equation*}
	\mathbf{A} = \sum_{k = 0}^d \bigotimes_{\ell = 0}^d A_h^{(k)}\left( \ell \right) \text{ where } A_h^{(k)}\left( \ell \right) = 
	\begin{cases} 			
	A_h^{(\ell)}& \text{if } \ell = d,\\
	\operatorname{diag}\left(p_h^{(\ell)}\right)& \text{if } \ell + k = d  \text{ and } k \neq 0, \\
	\Id_{n_{d-k}}& \text{otherwise}
	\end{cases}
\end{equation*}
with discrete parameters \(p_h^{(\ell)} = (p_h^{(\ell)}(1), \dots, p_h^{(\ell)}(n_{\ell})) \).
Similar results can be obtained for the \righthandside.

Concluding, we represent the operator and the right-hand side of~\eqref{eq:problem} exactly using low-rank tensor formats.
Further, we approximate the solution of~\eqref{eq:problem} using the hierarchical Tucker format in Algorithm~\ref{alg:pcg}.

%% file: tex/algorithms.tex
\section{The tensorized Metropolis-Hastings algorithm}\label{sec:algos}
Having proved the theory for our Bayesian inverse EMG problem and a low-rank tensor representation of the operator and right-hand side of the discrete forward EMG problem, we now derive our final main contribution: A fast tensorized Metropolis-Hastings algorithm.
Therefore, we combine the precomputation of the forward EMG problem described in Section~\ref{sec:EMG_model} for all parameters simultaneously using the hierarchical Tucker format and Algorithm~\ref{alg:pcg} with the Metropolis-Hastings sampling, as shown in Algorithm~\ref{algo:Metropolis}.
\begin{algorithm}
\caption{tensorized Metropolis-Hastings.}\label{algo:Metropolis}
\begin{algorithmic}[1]
\REQUIRE{Starting point \(p_{h,{(1)}}\) for the Markov chain, sampling radius \(\delta \)}
\ENSURE{A Markov chain \(p_h\)}
\STATE{Precompute \(\mathcal{G}(p_h)\) for all \(p_h \in \mathcal{J}_h\) using tensor formats}
\FOR{\(j = 1,\ldots,J-1  \)}\label{algo:Metropolis:line1}
\STATE{Propose \(\tilde{p}_h \sim \mathcal{U}([p_{h,{(j)}}-\delta,p_{h,{(j)}}+\delta] \cap \mathcal{J}_h) \) independent of \(p_{h,{(j)}}\)}
\STATE{Draw \(c \sim \mathcal{U}(0,1)\)}
\IF{\(c \le a(\mathcal{G}(p_{h,{(j)}}), \mathcal{G}(\tilde{p}_h))\)}\label{algo:Metropolis:line4}
\STATE{\(p_{h,{(j+1)}} = \tilde{p}_h\)} 
\ELSE{\STATE{\(p_{h,{(j+1)}} = p_{h,{(j)}}\)}}\ENDIF{}
\ENDFOR{}
\end{algorithmic}
\end{algorithm}

To be more precise, we first choose a fixed number of samples \(J \in \mathbb{N}\) that have to be drawn during the sampling process.
We then precompute the solution of the parameter-dependent forward EMG problem on a discrete set \(\mathcal{J}_h\) in the hierarchical Tucker format using the PCG method from Algorithm~\ref{alg:pcg} and store the data-sparse solution.
Recall that storing the solution of the parameter-dependent problem for all parameters is only feasible within data-sparse formats like the hierarchical Tucker format.

Doing so enables us to evaluate the precomputed tensor solution with arithmetic cost in \( \mathcal{O}(n d r^3) \) and evaluate this solution fast instead of solving the discretized forward EMG problem in every iteration in line~\ref{algo:Metropolis:line4} of the algorithm.
Note that we draw new samples \(\tilde{p}_h\) uniformly from an interval with radius \(\delta \) around the last accepted sample intersected with the discrete set \( \mathcal{J}_h \) to account for the local behavior of the potential \(\Phi \) and to accelerate convergence.

We assume that the cost of drawing one sample from the posterior distribution equals the solution time \(T_\text{s}\) of the discretized forward EMG problem for the standard Metropolis-Hastings algorithm and the evaluation time \(T_\text{e}\) of the precomputed tensor solution for the tensorized Metropolis-Hastings algorithm.
Thus, the runtime of the standard Metropolis-Hastings algorithm is \( J T_\text{s} \), while the runtime of the tensorized algorithm is the sum of the precomputation time \(T_\text{p}\) and the evaluation times, i.e., \( T_\text{p} + J T_\text{e} \).
We notice that asymptotically the speedup \( \frac{J T_\text{s}}{ T_\text{p} + J T_\text{e}} \) is limited by \( \frac{T_\text{s}}{T_\text{e}} \) for \(J \to \infty \).

Based on our mathematical theory we expect that the Markov chains constructed by both algorithms behave similarly.
This is due to the fact that we exactly represent the operator and the right-hand side of the forward EMG problem for all discrete parameter combinations within the hierarchical Tucker format.
Additionally, we compute the tensor solution using Algorithm~\ref{alg:pcg} with specified truncation accuracy, resulting in an error-controlled approximation.

%% file: tex/numerical_experiments.tex
\section{Numerical experiments}\label{sec:NumExp}
We illustrate our method for the inverse EMG problem with numerical experiments.
We conduct all experiments in \textsc{Matlab} using the KerMor~framework\footnote{\url{https://www.morepas.org/software/kermor/index.html}} and the \texttt{htucker}~toolbox~\cite{Kressner2014}.
Throughout our experiments we use the following default settings.

The geometry that we use is a muscle cuboid of size \SI{4 x 2 x 1}{\centi\meter} that is equipped with \num{30 x 30} muscle fibers. 
The muscle geometry is discretized using the grid size \(h_\text{M} = \frac{1}{3}\) while the muscle fibers are discretized using \(30\) grid points, and we use \(101\) time steps.
We fix the extracellular conductivity at \(\sigma_\text{e} = \operatorname{diag}(6.7, 6.7, 6.7)\).
As reference conductivity we choose \(p^\text{ref} = (\num{0.893}, \num{8.930}, \num{0.893})\), i.e., the muscle fiber direction is the second unit vector and the muscle fibers are aligned parallel to the second coordinate axis.
We allow the muscle fiber direction to be one of the three unit vectors.
As upper bound on the conductivity we define \(s_{+} = \num{10}\) and \(s_{-} = \num{0.001}\) as lower bound which we also set as the discretization step size in the parameter space, i.e., \(h_\sigma = s_{-}\).

For computing the tensor solution of~\eqref{eq:problem}, we use Algorithm~\ref{alg:pcg}.
There we set \(k_{\max} = 15\), \(\varepsilon = \num{1e-4}\) and we truncate to a relative accuracy of \num{1e-6}. 
As preconditioner we define \(\mathbf{M} \coloneqq \Id_{n_d} \otimes \dots \otimes \Id_{n_1} \otimes A_h^{(0)}\), since we observed similar convergence behavior and similar runtimes of the algorithm independent of the chosen low-rank tensor preconditioner, see, e.g.,~\cite{Kressner2011}, in former experiments.
We compute the tensor solution on a suitable conductivity grid with grid size \(h_\sigma \) and \(A_h^{(0)}\) using the conductivity at the midpoint of that grid.
For handling the time-dependency in the right-hand side, we solve the corresponding linear system for all time steps simultaneously.
This leads to a tensor of size \num{364 x 101 x 4000 x 4001 x 4000}.

For sampling from the posterior distribution of intracellular conductivity given EMG measurements, we use Algorithm~\ref{algo:Metropolis}. 
There we set the total number of samples to \num{500000} and use Gaussian noise with \(\xi = 2.0\). 
The algorithm draws a conductivity proposal in a sampling radius \( \delta = 1.5 \) around the last accepted sample.
As default we draw the initial guess from a uniform distribution on an interval with radius \(\delta \) around the reference solution, and we discard the first \num{200} samples as burn-in.
These choices proved reasonable in our parameter studies.
Additionally, we modify the algorithm such that it also samples the muscle fiber direction as one of the unit vectors.

We call Algorithm~\ref{algo:Metropolis} using the \textsc{Matlab} build-in QR decomposition to solve the forward problem for the proposed conductivity in each iteration the \emph{standard algorithm}~(SA), and we call Algorithm~\ref{algo:Metropolis} using the precomputed tensor solution the \emph{tensorized algorithm}~(TA).
\subsection*{Rank of the hierarchical Tucker format solution}
In our first numerical experiment, we examine the hierarchical Tucker rank, see Definition~\ref{def:HTRank}, of the tensor solution of the linear system to support our assumption that the solution is well approximated with low rank.
Further, a small rank is important for efficient arithmetic operations as some of these operations in low-rank tensor formats scale in \(\mathcal{O}(r^4)\), see Section~\ref{sec:repParameterProblems}.
Therefore, in Figure~\ref{fig:lowranksol} we show a logarithmic-linear plot of the relative singular values for the corresponding matricizations of the solution of the forward problem using our default setting.
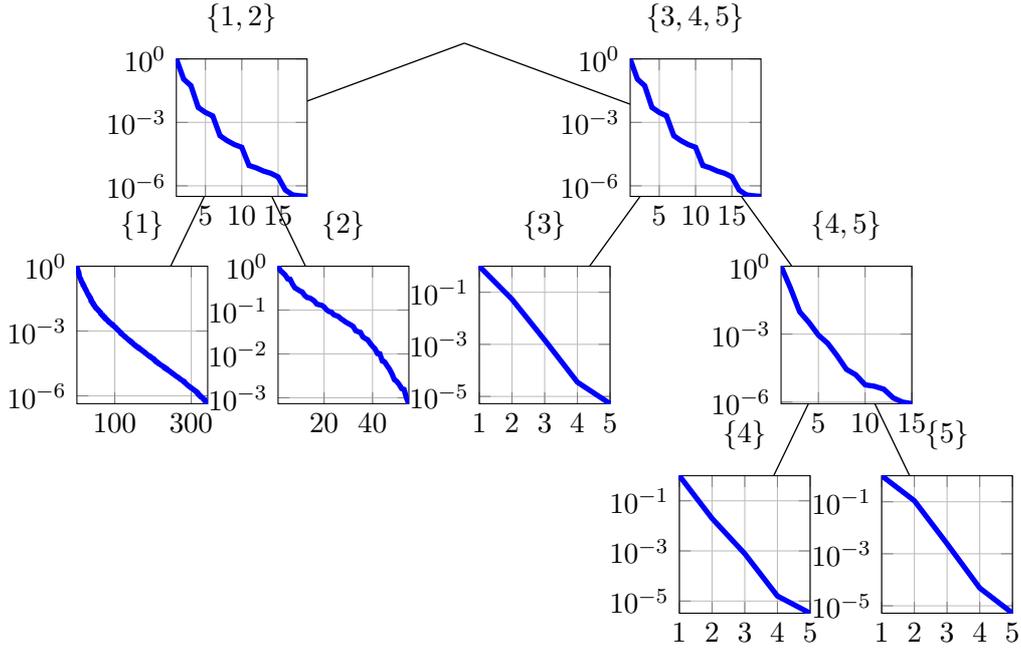
\begin{figure}
\centering
\input{img/exp1_save.tex} 
\caption{Relative singular values of the corresponding matricization of the low-rank solution of the forward EMG problem.}\label{fig:lowranksol}
\end{figure}
We observe that the rank of the matricization remains smaller than \num{6} in the parameter space, i.e., the rank of the matricizations corresponding to \(\lbrace 3 \rbrace \), \(\lbrace 4 \rbrace \), and \(\lbrace 5 \rbrace \).
We also see that the rank of the matricization corresponding to \( \lbrace 2 \rbrace \) is \num{55} while the rank of the matricization corresponding to \( \lbrace 1 \rbrace \) is \num{343}.
We expect that the matricization corresponding to \( \lbrace 1 \rbrace \) has full rank since this separates the spatial dimension, i.e., \( \lbrace 1 \rbrace \), and the time dimension, i.e., \( \lbrace 2 \rbrace \), and since each time step yields its own right-hand side.
Using tensor formats, we reduce the theoretical storage cost of the full tensor from \( \SI{1.88e+10}{\mega\byte} ( \approx \SI{18800000}{\giga\byte} ) \) to \SI{4.41}{\mega\byte} counting the storage cost for \(1\) entry as \SI{64}{\bit}.
\subsection*{Comparison of the tensorized algorithm and the standard algorithm}
For the validation of our tensorized algorithm, we compare its statistical behavior to the standard Metropolis-Hastings algorithm.
We run both algorithms in our default setting for the reference conductivities \(p^\text{ref}_1 = (0.893, 8.930, 0.893)\) and \(p^\text{ref}_2 = (0.893, 0.893, 8.930)\).

We present the acceptance rates~\( \frac{\#\text{samples acc.}}{\#\text{samples drawn}}\), the mean absolute deviations~(MADs) \( \frac{1}{\#\text{samples acc.}} \sum_{k=1}^{\#\text{samples acc.}} | p_{(k)} - \bar{p} | \) and variance \( \frac{1}{\#\text{samples acc.} - 1} \sum_{k=1}^{\#\text{samples acc.}} {(p_{(k)} - \bar{p})}^2 \) of the accepted diagonal entries of the conductivities in Table~\ref{tab:comparison_full_tensor}.
\begin{table}
\centering
\caption{Comparison of the standard algorithm~(SA) and the tensorized algorithm~(TA) with \num{500000} drawn samples, Gaussian noise with \(\xi = \num{2.0}\), and sampling radius \(\delta = \num{1.5}\).}\label{tab:comparison_full_tensor}
\begin{tabular}{@{}lllll@{}}
\toprule
& \multicolumn{2}{c}{\(p^\text{ref}_1\)} & \multicolumn{2}{c}{\(p^\text{ref}_2\)}\\
\cmidrule(l){2-5}
& SA & TA & SA & TA \\
\midrule
Acceptance rate~(\si{\percent}) & \num{5.39} & \num{5.38} & \num{5.79} & \num{5.79} \\
\( \operatorname{MAD}(p(1)) \) & \num{0.86} & \num{0.86} & \num{0.50} & \num{0.50} \\
\( \operatorname{MAD}(p(2)) \) & \num{0.44} & \num{0.44} & \num{0.24} & \num{0.24} \\
\( \operatorname{MAD}(p(3)) \) & \num{0.17} & \num{0.17} & \num{0.55} & \num{0.55} \\
\( \operatorname{Var}(p(1)) \) & \num{1.05} & \num{1.05} & \num{0.38} & \num{0.38} \\
\( \operatorname{Var}(p(2)) \) & \num{0.29} & \num{0.29} & \num{0.09} & \num{0.09} \\
\( \operatorname{Var}(p(3)) \) & \num{0.04} & \num{0.04} & \num{0.44} & \num{0.44} \\
\bottomrule
\end{tabular}
\end{table}
For the reference values \(p^\text{ref}_1\) and \(p^\text{ref}_2\) we observe that both methods have similar acceptance rates.
We further notice that the SA and TA have a comparable reliability, i.e., comparable MAD and variance. 

We conclude that the sampling process of both algorithms is similar and that our tensor approach is therefore a promising ansatz to accelerate the SA if the discretization error of the forward problem is small. 
In this case, we furthermore reason that our results indicate that the tensor solution of the forward EMG problem is indeed a good approximation to the solution that we obtain using the \textsc{Matlab} build-in QR decomposition.
We highlight that these results are in line with our theoretical findings from Section~\ref{sec:guiding_example}.

\subsection*{Speedup tests}
First, we examine the speedup \( \frac{\text{runtime SA}}{\text{runtime TA}}\) of our tensor method compared to the standard method for fixed discretization grid size \(h_\text{M}\) and varying number of samples.
Therefore, we run both algorithms in the default setting for \num{125} samples and double the number of samples until we reach \num{128000} samples.
We present the speedup of the TA compared to the SA in Figure~\ref{fig:runtime_fixed_grid}.

We observe that the speedup curve grows steadily and flattens as the number of samples increases.
This is due to the fact that the influence of the precomputation time of the TA, which is \( \bar{T}_\text{p} \approx \SI[round-mode=places,round-precision=2]{13.713409090909094}{\second}\) on average, decreases with growing number of samples.
As mentioned in Section~\ref{sec:algos}, the speedup is bounded by the quotient \(\frac{T_\text{s}}{T_\text{e}}\).
We insert the average time \(\bar{T}_\text{s} \approx \SI[round-mode=places,round-precision=4]{0.148134258852273}{\second}\) needed for one sample using the SA and the average time \(\bar{T}_\text{e} \approx \SI[round-mode=places,round-precision=4]{0.003741202008523}{\second}\) needed for one sample using the TA and obtain an upper bound of \num[round-mode=places,round-precision=2]{39.595364942818300} for the speedup.
For \num{128000} samples the speedup is \num[round-mode=places,round-precision=2]{36.7664517437885}, which corresponds to a runtime of \SI[round-mode=places,round-precision=2]{5.265730027778}{\hour} using the SA, compared to \( \SI[round-mode=places,round-precision=2]{0.14322105556}{\hour} ( \approx \SI[round-mode=places,round-precision=2]{8.593263333599999}{\minute} ) \) using the TA\@.

Further, we run both algorithms in the default setting for grid sizes \(h_\text{M} = \frac{1}{3}, \frac{1}{6}, \frac{1}{9}, \frac{1}{12} \).
Furthermore, to reduce the overall computation time, we reduce the number of samples to \num{100}.
We use our findings from Section~\ref{sec:algos} to extrapolate the measured sampling times to \num{100000} samples.
To be more precise, we first compute the average time for drawing one sample with both algorithms, then scale this number by \num{100000} to achieve estimates on \( T_\text{e} \) for the TA and \(T_\text{s} \) for the SA and add the measured precomputation time \( T_\text{p} \) for the TA\@.
Note that the precomputation time is independent of the number of samples.
Figure~\ref{fig:runtime_fixed_samples} shows the speedup resulting from this extrapolation.

\begin{figure}
\centering
\begin{minipage}{0.45\textwidth}
\centering
\begin{tikzpicture}
\begin{semilogxaxis}[grid=major, width=\linewidth,
xlabel={Number of samples}, ylabel={Speedup}]
\addplot table [x=number_of_samples,y=speed_up] {img/data_experiment_3_new.dat};
\end{semilogxaxis}
\end{tikzpicture}
\caption{Speedup of the tensorized algorithm compared to the standard algorithm for fixed grid size and a varying number of samples.}\label{fig:runtime_fixed_grid}
\end{minipage}
\hfill
\begin{minipage}{0.45\textwidth}
\centering
\begin{tikzpicture}
\begin{axis}[grid=major,width=\linewidth,
xlabel={Grid size \(h_\text{M}\)}, ylabel={Speedup},xtick={3,6,9,12},xticklabels={\(\frac{1}{3}\),\(\frac{1}{6}\),\(\frac{1}{9}\),\(\frac{1}{12}\)}]
\addplot table [x=h_inv,y=speedup] {img/data_experiment_4_estimated.dat};
\end{axis}
\end{tikzpicture}
\caption{Estimated speedup of the tensorized algorithm compared to the standard algorithm for varying grid size \(h_\text{M}\) and \num{100000} samples.}\label{fig:runtime_fixed_samples}
\end{minipage}
\end{figure}
As expected, we observe that the speedup in Figure~\ref{fig:runtime_fixed_samples} grows steadily and is unbounded in contrast to the speedup for fixed grid size and increasing number of samples.
For \( h_\text{M} = \frac{1}{12} \) we observe a speedup of \num[round-mode=places,round-precision=2]{650.8135031} which corresponds to a runtime of \SI[round-mode=places,round-precision=2]{2.865513472222}{\hour} for the TA compared to a runtime of \(\SI[round-mode=places,round-precision=2]{1864.91486111}{\hour} (\approx \SI[round-mode=places,round-precision=2]{77.70478009}{\day}) \) for the SA\@.

We expect that the TA outperforms the SA for realistic muscle geometries or fine grid sizes.
Furthermore, we conclude that using the TA enables us to solve problems that are infeasible to solve using the SA, in reasonable time.

%% file: img/exp1_save.tex
\begin{tikzpicture}

\begin{axis}[%
width=0.9\textwidth,
height=0.564\textwidth,
at={(0\textwidth,0\textwidth)},
scale only axis,
xmin=-0.0111524163568774,
xmax=0.988847583643123,
ymin=0,
ymax=1,
axis line style={draw=none},
ticks=none,
legend style={legend cell align=left, align=left, draw=white!15!black}
]
\addplot [color=black, line width=0.5pt]
  table[row sep=crcr]{%
0.425	0.966666666666667\\
0.2	0.833333333333333\\
};

\addplot [color=black, line width=0.5pt]
  table[row sep=crcr]{%
0.2	0.833333333333333\\
0.1	0.5\\
};

\addplot [color=black, line width=0.5pt]
  table[row sep=crcr]{%
0.65	0.833333333333333\\
0.5	0.5\\
};

\addplot [color=black, line width=0.5pt]
  table[row sep=crcr]{%
0.8	0.5\\
0.7	0.166666666666667\\
};

\addplot [color=black, line width=0.5pt]
  table[row sep=crcr]{%
0.425	0.966666666666667\\
0.65	0.833333333333333\\
};

\addplot [color=black, line width=0.5pt]
  table[row sep=crcr]{%
0.2	0.833333333333333\\
0.3	0.5\\
};

\addplot [color=black, line width=0.5pt]
  table[row sep=crcr]{%
0.65	0.833333333333333\\
0.8	0.5\\
};

\addplot [color=black, line width=0.5pt]
  table[row sep=crcr]{%
0.8	0.5\\
0.9	0.166666666666667\\
};

\end{axis}

\begin{axis}[%
width=0.117\textwidth,
height=0.124\textwidth,
at={(0.135\textwidth,0.407\textwidth)},
scale only axis,
xmin=1,
xmax=19,
ymode=log,
ymin=3.16947311476433e-07,
ymax=1,
yminorticks=false,
axis background/.style={fill=white},
title={\( \lbrace 1, 2 \rbrace \)},
xmajorgrids,
ymajorgrids,
legend style={legend cell align=left, align=left, draw=white!15!black}
]
\addplot [color=blue, line width=2.0pt]
  table[row sep=crcr]{%
1	1\\
2	0.110302342991466\\
3	0.0544933157655251\\
4	0.00508689467128234\\
5	0.00294931468752471\\
6	0.00202725556674593\\
7	0.000234119000613744\\
8	0.000135942901869391\\
9	8.90041847708852e-05\\
10	6.63487259078038e-05\\
11	9.12222462697883e-06\\
12	7.0712289552342e-06\\
13	4.97153613644241e-06\\
14	3.91686814606653e-06\\
15	2.6429193442321e-06\\
16	6.31694055716076e-07\\
17	3.70769800935048e-07\\
18	3.37097725381317e-07\\
19	3.16947311476433e-07\\
};

\end{axis}

\begin{axis}[%
width=0.117\textwidth,
height=0.124\textwidth,
at={(0.541\textwidth,0.407\textwidth)},
scale only axis,
xmin=1,
xmax=19,
ymode=log,
ymin=3.16947311476433e-07,
ymax=1,
yminorticks=false,
axis background/.style={fill=white},
title={\( \lbrace 3, 4, 5 \rbrace \)},
xmajorgrids,
ymajorgrids,
legend style={legend cell align=left, align=left, draw=white!15!black}
]
\addplot [color=blue, line width=2.0pt]
  table[row sep=crcr]{%
1	1\\
2	0.110302342991466\\
3	0.0544933157655251\\
4	0.00508689467128234\\
5	0.00294931468752471\\
6	0.00202725556674593\\
7	0.000234119000613744\\
8	0.000135942901869391\\
9	8.90041847708852e-05\\
10	6.63487259078038e-05\\
11	9.12222462697884e-06\\
12	7.0712289552342e-06\\
13	4.97153613644241e-06\\
14	3.91686814606653e-06\\
15	2.6429193442321e-06\\
16	6.31694055716076e-07\\
17	3.70769800935048e-07\\
18	3.37097725381317e-07\\
19	3.16947311476433e-07\\
};

\end{axis}

\begin{axis}[%
width=0.117\textwidth,
height=0.124\textwidth,
at={(0.046\textwidth,0.22\textwidth)},
scale only axis,
xmin=1,
xmax=343,
ymode=log,
ymin=4.34896514706366e-07,
ymax=1,
xtick={100,300},
yminorticks=false,
axis background/.style={fill=white},
title={\( \lbrace 1 \rbrace \)},
xmajorgrids,
ymajorgrids,
legend style={legend cell align=left, align=left, draw=white!15!black}
]
\addplot [color=blue, line width=2.0pt]
  table[row sep=crcr]{%
1	1\\
2	0.789460263935602\\
3	0.709126787475431\\
4	0.62264333617367\\
5	0.512821767314359\\
6	0.501361075835153\\
7	0.391922048619632\\
8	0.318074173849051\\
9	0.29100609209015\\
10	0.273037807905443\\
11	0.256912697959509\\
12	0.215486757267251\\
13	0.193901472998465\\
14	0.189094291164406\\
15	0.176316298742687\\
16	0.160186672419378\\
17	0.137372801013818\\
18	0.135840684114104\\
19	0.129323250658367\\
20	0.118859707981073\\
21	0.104831547792493\\
22	0.0940315313223463\\
23	0.0905661105715456\\
24	0.0814475747083721\\
25	0.0786089386253424\\
26	0.0717709218314365\\
27	0.0641551570170433\\
28	0.0591760788776618\\
29	0.0535657613092984\\
30	0.0502822479327499\\
31	0.0478676231787077\\
32	0.0441033689027431\\
33	0.0432717973230458\\
34	0.038676371113238\\
35	0.0362943619612895\\
36	0.0328859308552686\\
37	0.0323073849835435\\
38	0.0257809612768738\\
39	0.0254454671788144\\
40	0.0230955551120924\\
41	0.0214556340785717\\
42	0.020988275149582\\
43	0.0186319838732136\\
44	0.0184723826700505\\
45	0.0177382280789009\\
46	0.0153511202133951\\
47	0.014879712858905\\
48	0.0138634999569473\\
49	0.0127742177475567\\
50	0.0124970866125385\\
51	0.011694599569202\\
52	0.0116495442193402\\
53	0.0110586943712687\\
54	0.010448037116966\\
55	0.0100094597223388\\
56	0.00931485997638852\\
57	0.00921194431189519\\
58	0.00898343437209783\\
59	0.00836534910324429\\
60	0.00790539332479659\\
61	0.00759744579765413\\
62	0.00712633376851513\\
63	0.00703440688117806\\
64	0.00630678218776989\\
65	0.00595227795010821\\
66	0.00582407724003193\\
67	0.00553870745138533\\
68	0.00546326079269176\\
69	0.00537605676547556\\
70	0.00482556878298698\\
71	0.00452483176770978\\
72	0.00440356656003256\\
73	0.00431421590002924\\
74	0.00428491626433227\\
75	0.00379880168824663\\
76	0.00375939525323591\\
77	0.00361665394501482\\
78	0.00357919062894503\\
79	0.00344146775929738\\
80	0.00332041193713999\\
81	0.0031611271566538\\
82	0.00292335042451553\\
83	0.00288490519772555\\
84	0.00271900616075621\\
85	0.00257592667142663\\
86	0.00255963195709969\\
87	0.0024563234546761\\
88	0.00239866288807346\\
89	0.00229704465558182\\
90	0.00220024054727112\\
91	0.00217940715975538\\
92	0.00212102345467972\\
93	0.00207632087255331\\
94	0.00192654788020104\\
95	0.00189678785307062\\
96	0.00177275980135782\\
97	0.00174324445528955\\
98	0.00169612955538386\\
99	0.00164946492233551\\
100	0.00154957141409581\\
101	0.00147588394720197\\
102	0.00144344196322671\\
103	0.00142390690782714\\
104	0.00132876900295746\\
105	0.00126429143195622\\
106	0.00122366193131596\\
107	0.00119830689789942\\
108	0.00111949404772308\\
109	0.00107572598769618\\
110	0.00105718193977519\\
111	0.00101991089034624\\
112	0.000968067553456788\\
113	0.000930921489052561\\
114	0.000907276570165399\\
115	0.00086712394718604\\
116	0.000784055385509549\\
117	0.000773332157343536\\
118	0.000740109201734547\\
119	0.000730843752878467\\
120	0.000691663958546233\\
121	0.000666071624854116\\
122	0.000647502551107078\\
123	0.000620822778477453\\
124	0.00060311818273263\\
125	0.000587063688649531\\
126	0.000569993848348422\\
127	0.000528473090984564\\
128	0.000516915876431763\\
129	0.000501886932167133\\
130	0.000490703225389268\\
131	0.000470331089688126\\
132	0.000458051834203827\\
133	0.000436717404891535\\
134	0.000417804045234222\\
135	0.000405913971378553\\
136	0.000400154933549092\\
137	0.000381444026793843\\
138	0.00036358113987651\\
139	0.000361051209203965\\
140	0.000346989092969447\\
141	0.000338426228079741\\
142	0.000325759422943359\\
143	0.00032332124148185\\
144	0.00029406877416345\\
145	0.000287316358485338\\
146	0.000278877334780218\\
147	0.000269735867478367\\
148	0.000266900893620699\\
149	0.000254497944840457\\
150	0.000239780544152279\\
151	0.000231503535869149\\
152	0.000226790816671952\\
153	0.000222054974314869\\
154	0.000211234748675814\\
155	0.000209175861624547\\
156	0.000196804228682884\\
157	0.000192301007709021\\
158	0.000186103737322028\\
159	0.000183954680337434\\
160	0.000175908562876061\\
161	0.000173882506429151\\
162	0.000172237134911951\\
163	0.000161387712943487\\
164	0.000158023172430642\\
165	0.000156735685536598\\
166	0.000147137659442725\\
167	0.00014321927567131\\
168	0.000138122773873042\\
169	0.000136229657759393\\
170	0.000135529433633675\\
171	0.000130946843030769\\
172	0.000125542951266451\\
173	0.000119612671937577\\
174	0.00011765408914201\\
175	0.000112417958426651\\
176	0.000110464399272443\\
177	0.000108027099673934\\
178	0.00010395187764875\\
179	9.84225427138587e-05\\
180	9.72185389297937e-05\\
181	9.35918654737772e-05\\
182	9.0683778764907e-05\\
183	8.69367865126313e-05\\
184	8.47441247683526e-05\\
185	8.29792123099998e-05\\
186	7.8624230838619e-05\\
187	7.60599435607466e-05\\
188	7.40116730709616e-05\\
189	7.21648197725392e-05\\
190	7.11752303502899e-05\\
191	7.02994037796953e-05\\
192	6.61680224011532e-05\\
193	6.52892554327539e-05\\
194	6.19018510598311e-05\\
195	6.03750249904904e-05\\
196	5.98975213535942e-05\\
197	5.76387511919078e-05\\
198	5.66446893866012e-05\\
199	5.45549422417133e-05\\
200	5.24696749439145e-05\\
201	5.11132927344008e-05\\
202	5.0968556608962e-05\\
203	4.70241822601379e-05\\
204	4.55172954007331e-05\\
205	4.349124187641e-05\\
206	4.27035572847661e-05\\
207	3.95030904517992e-05\\
208	3.89376116613335e-05\\
209	3.81768892053511e-05\\
210	3.75950532848953e-05\\
211	3.49400683171179e-05\\
212	3.43180717704231e-05\\
213	3.36128783043866e-05\\
214	3.28219371837368e-05\\
215	3.08665348125865e-05\\
216	3.05359817595015e-05\\
217	2.94830640288144e-05\\
218	2.87676137779925e-05\\
219	2.82859866540327e-05\\
220	2.72265020830787e-05\\
221	2.64529907547224e-05\\
222	2.58217030515025e-05\\
223	2.51684079415714e-05\\
224	2.47640015781058e-05\\
225	2.39318878830633e-05\\
226	2.29339047900319e-05\\
227	2.26371710950123e-05\\
228	2.20560457672606e-05\\
229	2.09553288905095e-05\\
230	2.04001649255255e-05\\
231	2.02550756891609e-05\\
232	1.86231428343805e-05\\
233	1.82796907363502e-05\\
234	1.79863989820255e-05\\
235	1.78946884484323e-05\\
236	1.71573453821621e-05\\
237	1.70562997807396e-05\\
238	1.61360376904479e-05\\
239	1.57065682074835e-05\\
240	1.53124793397608e-05\\
241	1.49183297451406e-05\\
242	1.47015414623387e-05\\
243	1.39737793313972e-05\\
244	1.3388333850047e-05\\
245	1.30494415203213e-05\\
246	1.28594493190638e-05\\
247	1.2510576925223e-05\\
248	1.21851411105847e-05\\
249	1.15742999110537e-05\\
250	1.13083782278712e-05\\
251	1.07927620287355e-05\\
252	1.04971569829916e-05\\
253	1.04367478332609e-05\\
254	1.01936452034278e-05\\
255	9.83623188984923e-06\\
256	9.37216462680827e-06\\
257	9.11375041418646e-06\\
258	8.72078988059363e-06\\
259	8.5136154451487e-06\\
260	8.3021037724461e-06\\
261	8.02923852852974e-06\\
262	7.81796366149477e-06\\
263	7.67602365249249e-06\\
264	7.46297350336354e-06\\
265	7.27591592401217e-06\\
266	7.20959546334567e-06\\
267	6.9531757156801e-06\\
268	6.70643550018377e-06\\
269	6.48665429238807e-06\\
270	6.23452996566684e-06\\
271	6.00142592334905e-06\\
272	5.89726489296348e-06\\
273	5.63897955766859e-06\\
274	5.49456084162613e-06\\
275	5.41405519270694e-06\\
276	5.26418858164157e-06\\
277	5.11681290325535e-06\\
278	4.9996348199568e-06\\
279	4.86548014352922e-06\\
280	4.78300499475108e-06\\
281	4.4394411555411e-06\\
282	4.39188154287061e-06\\
283	4.3082564640564e-06\\
284	3.97866077804884e-06\\
285	3.81064075070695e-06\\
286	3.65909832516496e-06\\
287	3.56553961394979e-06\\
288	3.46392485180626e-06\\
289	3.36858619318634e-06\\
290	3.14008058817055e-06\\
291	3.0149208631544e-06\\
292	2.94762935399919e-06\\
293	2.89431500239965e-06\\
294	2.80329335210566e-06\\
295	2.73895706436936e-06\\
296	2.61605317773283e-06\\
297	2.55966062249158e-06\\
298	2.46949228137546e-06\\
299	2.37078304287541e-06\\
300	2.34219077490744e-06\\
301	2.24760320401202e-06\\
302	2.20862395777012e-06\\
303	2.11791828251372e-06\\
304	2.03486057600869e-06\\
305	1.99873989433245e-06\\
306	1.92504363362225e-06\\
307	1.88736932710593e-06\\
308	1.81696454871977e-06\\
309	1.79260881041104e-06\\
310	1.74098417771272e-06\\
311	1.65809681823992e-06\\
312	1.59962308854844e-06\\
313	1.58132153732161e-06\\
314	1.48741369145339e-06\\
315	1.4189029334229e-06\\
316	1.40678588668007e-06\\
317	1.29396154347302e-06\\
318	1.26659519190573e-06\\
319	1.18531465925652e-06\\
320	1.16381598157486e-06\\
321	1.05846674852064e-06\\
322	1.0171336521952e-06\\
323	9.82569627128338e-07\\
324	9.75885237550155e-07\\
325	9.25282624519821e-07\\
326	9.17571493278535e-07\\
327	8.68754979562357e-07\\
328	8.38836231766759e-07\\
329	8.07419820591628e-07\\
330	7.59533726474921e-07\\
331	7.51991262262484e-07\\
332	7.30448480802337e-07\\
333	6.97270328800514e-07\\
334	6.43991646332227e-07\\
335	6.31411704895452e-07\\
336	6.27835524146544e-07\\
337	5.68131533271109e-07\\
338	5.19558780615421e-07\\
339	5.12477210977082e-07\\
340	4.97234038802923e-07\\
341	4.74280464973417e-07\\
342	4.46315880071566e-07\\
343	4.34896514706366e-07\\
};

\end{axis}

\begin{axis}[%
width=0.117\textwidth,
height=0.124\textwidth,
at={(0.226\textwidth,0.22\textwidth)},
scale only axis,
xmin=1,
xmax=55,
ymode=log,
ymin=0.000734573699448516,
ymax=1,
yminorticks=false,
axis background/.style={fill=white},
title={\( \lbrace 2 \rbrace \)},
xmajorgrids,
ymajorgrids,
legend style={legend cell align=left, align=left, draw=white!15!black}
]
\addplot [color=blue, line width=2.0pt]
  table[row sep=crcr]{%
1	1\\
2	0.790804293501648\\
3	0.70855635673624\\
4	0.622424996891996\\
5	0.513797358094791\\
6	0.502735621266239\\
7	0.387213321533916\\
8	0.319418889099814\\
9	0.296944737653851\\
10	0.273330306313675\\
11	0.257667555684989\\
12	0.215074976336182\\
13	0.193741327309393\\
14	0.189064568516642\\
15	0.176340081077492\\
16	0.160318627057231\\
17	0.13646459095507\\
18	0.132332131917492\\
19	0.128667952033215\\
20	0.118211436552347\\
21	0.101505398461289\\
22	0.0927573193751783\\
23	0.0901514284991434\\
24	0.0781919777046827\\
25	0.0746632858657015\\
26	0.0709209048004119\\
27	0.0631419417117412\\
28	0.0582417683798092\\
29	0.052921903596298\\
30	0.0493693180091989\\
31	0.0459843281802623\\
32	0.0423403318182556\\
33	0.0339551298674909\\
34	0.0322139771833845\\
35	0.0310848815389065\\
36	0.0244191521982184\\
37	0.0224601771062845\\
38	0.0210174484279832\\
39	0.0175081059216292\\
40	0.0149433520743798\\
41	0.0131116953318419\\
42	0.0102357008893459\\
43	0.0100695470824475\\
44	0.00692267481296261\\
45	0.00656922619610235\\
46	0.00538522020814741\\
47	0.00450690639247702\\
48	0.00352460973033292\\
49	0.00258982971484066\\
50	0.00236364876246173\\
51	0.00200508627581894\\
52	0.00162416498352672\\
53	0.00155381350190843\\
54	0.000983744140416809\\
55	0.000734573699448516\\
};

\end{axis}

\begin{axis}[%
width=0.117\textwidth,
height=0.124\textwidth,
at={(0.406\textwidth,0.22\textwidth)},
scale only axis,
xmin=1,
xmax=5,
ymode=log,
ymin=5.27367451431292e-06,
ymax=1,
yminorticks=false,
axis background/.style={fill=white},
title={\( \lbrace 3 \rbrace \)},
xmajorgrids,
ymajorgrids,
legend style={legend cell align=left, align=left, draw=white!15!black}
]
\addplot [color=blue, line width=2.0pt]
  table[row sep=crcr]{%
1	1\\
2	0.053918139128565\\
3	0.00146848722550196\\
4	3.52642157338833e-05\\
5	5.27367451431292e-06\\
};

\end{axis}

\begin{axis}[%
width=0.117\textwidth,
height=0.124\textwidth,
at={(0.676\textwidth,0.22\textwidth)},
scale only axis,
xmin=1,
xmax=15,
ymode=log,
ymin=8.45012041815205e-07,
ymax=1,
yminorticks=false,
axis background/.style={fill=white},
title={\( \lbrace 4, 5 \rbrace \)},
xmajorgrids,
ymajorgrids,
legend style={legend cell align=left, align=left, draw=white!15!black}
]
\addplot [color=blue, line width=2.0pt]
  table[row sep=crcr]{%
1	1\\
2	0.110213315793759\\
3	0.00962723789934061\\
4	0.00324030590367269\\
5	0.000900115775707791\\
6	0.000390869693967491\\
7	0.000111469482662063\\
8	2.76710361509759e-05\\
9	1.59626288679487e-05\\
10	5.73895907701329e-06\\
11	4.96250900159402e-06\\
12	3.68091611477488e-06\\
13	1.52375056286734e-06\\
14	9.94502157334678e-07\\
15	8.45012041815205e-07\\
};

\end{axis}

\begin{axis}[%
width=0.117\textwidth,
height=0.124\textwidth,
at={(0.585\textwidth,0.031\textwidth)},
scale only axis,
xmin=1,
xmax=5,
ymode=log,
ymin=3.27966034011536e-06,
ymax=1,
yminorticks=false,
axis background/.style={fill=white},
title={\( \lbrace 4 \rbrace \)},
xmajorgrids,
ymajorgrids,
legend style={legend cell align=left, align=left, draw=white!15!black}
]
\addplot [color=blue, line width=2.0pt]
  table[row sep=crcr]{%
1	1\\
2	0.0197649054226217\\
3	0.000775766498122038\\
4	1.5850946641589e-05\\
5	3.27966034011536e-06\\
};

\end{axis}

\begin{axis}[%
width=0.117\textwidth,
height=0.124\textwidth,
at={(0.766\textwidth,0.031\textwidth)},
scale only axis,
xmin=1,
xmax=5,
ymode=log,
ymin=5.20371362244574e-06,
ymax=1,
yminorticks=false,
axis background/.style={fill=white},
title={\( \lbrace 5 \rbrace \)},
xmajorgrids,
ymajorgrids,
legend style={legend cell align=left, align=left, draw=white!15!black}
]
\addplot [color=blue, line width=2.0pt]
  table[row sep=crcr]{%
1	1\\
2	0.109300111772372\\
3	0.00243198584908193\\
4	4.82278986052632e-05\\
5	5.20371362244574e-06\\
};

\end{axis}

\begin{axis}[%
width=0.9\textwidth,
height=0.564\textwidth,
at={(0\textwidth,0\textwidth)},
scale only axis,
xmin=0,
xmax=1,
ymin=0,
ymax=1,
axis line style={draw=none},
ticks=none,
axis x line*=bottom,
axis y line*=left,
legend style={legend cell align=left, align=left, draw=white!15!black}
]
\end{axis}
\end{tikzpicture}%

%% file: tex/related_work.tex
\section{Related work}\label{sec:relatedwork}
Surface EMG signals have been used to localize the innervation zones of skeletal muscle, see, e.g.,~\cite{Doel2011} and the references therein.
Furthermore, researchers are interested in denoising surface EMG signals, i.e., in reducing crosstalk of neighboring muscles or neighboring muscle regions, see~\cite{Mesin2020a}.
In~\cite{Mesin2020}, regularization methods for inverse problems are used to reduce crosstalk in surface EMG signals.

To overcome the ill\hyp{}posedness of inverse problems, regularization methods like the Tikhonov regularization are a widely used ansatz, see, e.g.,~\cite{R13} and the references therein.
The Tikhonov regularization was used in~\cite{R18,R17} to reconstruct the electrical conductivity of biological tissue from EMG measurements.
Moreover, in~\cite{R17} model order reduction was used to accelerate the computations.
For other Bayesian inverse problems different approaches to speedup the sampling process have been examined, e.g., in~\cite{R3} the authors used a method based on polynomial chaos expansions to construct a surrogate of the forward problem.
In~\cite{R5} quasi Monte Carlo methods and multilevel Monte Carlo methods were used to accelerate the convergence of the sampling algorithm.

Furthermore, low-rank tensor methods were examined in the context of Bayesian inverse problems.
In~\cite{Dolgov2020} low-rank tensor formats were used to compute a surrogate of the target distribution.
There the authors directly approximated the target distribution using a generalization of the cross approximation.
In~\cite{Eigel2018} the authors used low-rank tensor formats to approximate the stochastic Galerkin solution of the parameter-dependent forward problem to achieve a discrete representation of the posterior distribution.

%% file: tex/conclusion.tex
\section{Conclusion}\label{sec:conclusion}%
Applying mathematical theory results in an efficient algorithm to solve the Bayesian inverse EMG problem.
Proving the well-posedness of the Bayesian inverse EMG problem guarantees the convergence of this algorithm.
Further, proving a data-sparse representation of the forward EMG problem allows for the efficient precomputation of the parameter-dependent forward solution.
The presented numerical experiments support this mathematical theory but also indicate that a high number of samples is required to obtain accurate results.
The sampling algorithm which uses the data-sparse representation of the forward EMG problem computes this high number of samples in a reasonable time.

The mathematical theory of the Bayesian inverse problem holds for general symmetric positive definite conductivities and thus for arbitrary muscle fiber directions.
The low-rank representation of the forward EMG problem, however, holds for fixed muscle fiber directions only.
In our numerical experiments the speedup using tensor methods enables solving problems with grid sizes that are infeasible using classical methods.
Therefore, future work is the generalization of the low-rank representation of the right-hand side of the forward EMG problem to arbitrary muscle fiber directions.
This generalization could enable the computation of realistic problems in medical applications and lead to a non-invasive and radiation-free imaging method.